\documentclass[mathpazo]{cicp}

\usepackage{amsmath}
\usepackage{graphicx}
\usepackage{mathrsfs}
\usepackage{bm}
\usepackage{float}
\usepackage{amsfonts,amssymb}
\usepackage{dsfont}
\usepackage{pifont}
\usepackage{tikz}
\usepackage{wrapfig} 
\usepackage{hyperref}
\usepackage{multirow}
\usepackage{mathtools}
\usepackage{subfigure}
\usepackage{ulem}
\usepackage{color}
\definecolor{commentcolor}{RGB}{85,139,78}
\definecolor{stringcolor}{RGB}{206,145,108}
\definecolor{keywordcolor}{RGB}{34,34,250}
\definecolor{backcolor}{RGB}{220,220,220}

\usepackage{accsupp}	
\newcommand{\emptyaccsupp}[1]{\BeginAccSupp{ActualText={}}#1\EndAccSupp{}}
\usepackage{listings}
\newtheorem{rremark}{Remark}[section]

\def\diam{\operatornamewithlimits{diam}}

\def\D{{\mathcal D}}
\def\E{{\mathcal E}}

\def\w{{\mathscr W}}
\def\f{{\mathscr F}}

\def\3bar{{|\!|\!|}}

\newtheorem{FD-algorithm}{5-Point Finite Difference Algorithm}[section]

\usepackage{colortbl}
\definecolor{mygray}{gray}{.9}
\definecolor{mypink}{rgb}{.99,.91,.95}
\definecolor{mycyan}{cmyk}{.3,0,0,0}
\newcommand{\vertiii}[1]{{\left\vert\kern-0.25ex\left\vert\kern-0.25ex\left\vert #1
    \right\vert\kern-0.25ex\right\vert\kern-0.25ex\right\vert}}
\begin{document}
\title{VPVnet: a velocity-pressure-vorticity neural network method for the Stokes' equations under reduced regularity}

\author[Y.~Liu and C.~Yang]{Yujie Liu\affil{1}
       and Chao Yang\affil{2}\comma\corrauth}
\address{\affilnum{1}\ Artificial Intelligence Research Center, Peng Cheng Laboratory, Shenzhen 518005, China. \\
         \affilnum{2}\ School of Mathematical Sciences, Peking University, Beijing 100871, China.}
\emails{{\tt liuyj02@pcl.ac.cn} (Y.~Liu), {\tt chao\_yang@pku.edu.cn} (C.~Yang)}


\begin{abstract}
We present VPVnet, a deep neural network method for the Stokes' equations under reduced regularity.
Different with recently proposed deep learning methods \cite{DGM_2020,raissi2019physics} which are based on the original form of PDEs, VPVnet uses the least square functional of the first-order velocity-pressure-vorticity (VPV) formulation (\cite{Bo-Nan_1990}) as loss functions.
As such, only first-order derivative is required in the loss functions, hence the method is applicable to a much larger class of problems, e.g. problems with non-smooth solutions.
Despite that several methods have been proposed recently to reduce the regularity requirement by transforming the original problem into a corresponding variational form,
while for the Stokes' equations,
the choice of approximating spaces for the velocity and the pressure has to satisfy the LBB condition additionally.
Here by making use of the VPV formulation, lower regularity requirement is achieved with no need for considering the LBB condition.
Convergence and error estimates have been established for the proposed method.
It is worth emphasizing that the VPVnet method is divergence-free and pressure-robust, while classical inf-sup stable mixed finite elements for the Stokes'
equations are not pressure-robust.
Various numerical experiments including 2D and 3D lid-driven cavity test cases are conducted to demonstrate its efficiency and accuracy.
\end{abstract}


\maketitle

\section{Stokes' equations}
Recently, the deep neural network (DNN) methods have attracted remarkable attention in the field of computational fluid dynamics  \cite{Brunton_2019,NSFnetLiHuiKarniadakis2020,Thuerey_2020,WangZhangCai_2020}.
In contrast to classical methods such as finite element, finite difference, and finite volume, the DNN approach does not require a mesh topology and can achieve good accuracy even when the domain of interest is presented by scattered discrete points \cite{Mehrkanoon_2015}.
Moreover, it requires less number of parameters to achieve the same accuracy as with finite element method on the grid points\cite{LagarisLikas1998}.
And it can lessen or even overcome the curse of dimensionality for high-dimensional problems \cite{E_MLM_2021},
it also has great potential in nonlinear approximations \cite{raissi2019physics,DGM_2018}.
However, the performance of existing DNN methods may degrade for the Navier-Stokes equations with low regularity,
e.g. when sharp local gradients present in a broad computational domain \cite{Ranade_CMAME_2021,Lee_ERA_2021}.
Noting that most DNN methods such as the recently developed physics informed neural network (PINN) method \cite{NSFnetLiHuiKarniadakis2020} employs
the residual of equations as the loss function, which requires derivatives of variables that is only applicable to problems whose solutions are sufficiently smooth,
for example $(\bm{u}, p)$ at least in $[H^2(\Omega)]^2\times H^1(\Omega)$ for the Stokes equations,
and is not able to cope properly with problems that have nonsmooth solutions. As exactly solving the Stokes's equations with low regularity is the preliminary for many Navier-Stokes applications such as the lid driven cavity problems and other problems with singular sources and sharp interfaces \cite{ArbogastBrunson_2007,CodinaSoto_1997,Bo-Nan_1990}, etc.
In this paper, we propose a deep neural network method for the Stokes' equations with reduced regularity, i.e. $(\bm{u}, p) \in [H^1(\Omega)]^2\times H^1(\Omega)$.

For simplicity, consider the following Stokes' problem which seeks a velocity field $\bm{u}$ and a pressure unknown $p$ satisfying
\begin{equation}\label{EQ:Stokes}
\left \{\begin{split}
 -\nu\Delta\bm{u} +\nabla p &= \bm{f},\qquad {\rm in}\  \Omega,\\
\nabla \cdot \bm{u}&=0, \qquad {\rm in}\ \Omega,\\
\bm{u} &= \bm{g}, \qquad {\rm on}\ \Gamma,\\
\end{split}\right.
\end{equation}
where $\Omega$ is a bounded polygonal domain with boundary $\Gamma = \partial \Omega$ in $d$ ($d=2$ or $3$) dimension, $\nu>0$ is a constant viscosity parameter, $\bm{f} = (f^{(1)}, \cdots ,f^{(d)})$ represents external source, $\bm{g}$ is a given Dirichlet boundary condition, and the pressure $p$ is assumed to have mean value zero; i.e., $\int_{\Omega}p d\Omega = 0$.
Noting that the classical solution of \eqref{EQ:Stokes} satisfies $\bm{u}\in [C^2(\Omega)\cap C^0(\Omega)]^d$ and $p\in C^1(\Omega)$.

Various deep learning based methods have been proposed and {explored} recently for the simulation of partial differential equations,
such as the least square methods \cite{DissanayakePhanThien_1994,Caizhiqiang2020}, the deep Ritz methods \cite{E_DRM_2018}, the {physics-informed neural
network (PINN) methods} \cite{raissi2019physics,raissi2020HPMScience,NSFnetLiHuiKarniadakis2020} and the variational formulation based methods \cite{Kharazmi_VPINN_2019,Khodayi-MehrZavlano-VarNet2020} among many others \cite{Caiwei_2020,LiKovachki_2020,Iskhakov2020_1,Iskhakov2020_2,DGM_2018}.
Here we roughly classify  {these} methods according to their formulation of the loss functions{;} the reader is referred to \cite{BeckHutzenthalerJentzenKuckuck_2020,Khodayi-MehrZavlano-VarNet2020} and the references therein for a more detailed review.
The least square methods may trace back to the {1990s.}, Dissanayake and Phan-Thhien \cite{DissanayakePhanThien_1994} have proposed training neural networks via a least square functional that based on the original formulation of the PDEs. Thereafter the method have been further developed and it has been shown that the sampling points can be obtained by a random sampling \cite{JinchaoXu_2018}, which is beneficial for high dimensional problems. Despite these advantages, for a {second-order} PDE, the minimization of the loss function over admissible functions leads to a fourth-order PDE, which is a more difficult problem than the original one \cite{Caizhiqiang2020}. Moreover, the method is only applicable to problems whose solutions are sufficiently smooth, more precisely, at least in $H^2(\Omega)$.
Several methods have been recently proposed by transform{ing} the original PDEs to the corresponding weak form so as to lower the regularity requirement.
For example, the deep Ritz method \cite{E_DRM_2018} uses the energy functional of the underlying PDEs as loss functions,
hence only first-order spatial derivatives is required for the second{-}order elliptic PDEs.
However, the method only applies to problems that have a underlying minimization principle.
{In addition to that}, variational type neural network methods \cite{Kharazmi_VPINN_2019,Khodayi-MehrZavlano-VarNet2020} have been proposed.
These methods train the neural network by minimizing variational residuals of PDEs over a set of test functions with their compact supports located at different regions in space and time. The loss function is expected to be discretization free and highly parallelizable \cite{Khodayi-MehrZavlano-VarNet2020} and the methods also employ lower order derivatives.
For instance, numerical experiments have shown good performance of the variational type methods for advection-diffusion problems with high Peclet numbers.

As for the Stokes problems, which is the main concern of this paper, different deep learning methods based on the least square functional of the original PDEs have been proposed and verified in \cite{DGM_2020,Caiwei_2020,raissi2019physics}.
As mentioned previously, these methods are only applicable to problems whose solutions are sufficiently smooth, for $(\bm{u}, p)$ at least in $[H^2(\Omega)]^2\times H^1(\Omega)$.
Since the Stokes' problem does not have an underlying minimization principle,  {methods such as} the deep Ritz method  {do} not apply.
As for the variational type methods, it is well known that the variational formulation of the Stokes' problem leads to a saddle-point problem.
Consequently, the combination of velocity and pressure interpolations is required to satisfy the famous Ladyzhenskaya-Babu$\check{s}$ka-Brezzi (LBB) condition \cite{GiraultRaviart_1986,LiRuo_2019}.
On the other side, by using different activation functions such as Sigmoid, deep neural network can give rise to a very wide range of functional classes that can be drastically different from the piecewise polynomial function classes used in classic{al} finite element methods.
These function classes, however, do not usually form a linear vector space \cite{JinchaoXu_2020} which makes it tricky to apply the variational type neural network methods to  Stokes' equations.

In this paper, we propose a neural network method that {is} based on the first{-}order velocity-pressure-vorticity {(VPV)} system \cite{Bo-Nan_1990,Bo-Nan_1992,Bo-Nan_1994}, namely VPVnet, to approximate the solution of the Stokes' equations.
The least square functional of the VPV system leads to a minimization problem rather
than a saddle point problem, thus the combination of velocity, vorticity and pressure interpolations is not subject to the restriction of the LBB conditions.
{And} no artificial boundary conditions need to be devised for the vorticity.
Moreover, the VPV least square functional only requires first-order spatial derivatives, hence it has less regularity requirements on the solutions and is applicable to a much larger class of problems compared with least square methods that based on the original form of the Stokes' equations.
Note that similar ideas of making use of first order systems has also been used for the second-order elliptic PDEs {and verified with one dimensional test cases in \cite{Caizhiqiang2020}, but have not been studied for more complex systems such as the Stokes' equations.}
Wang \textit{et. al.} \cite{WangZhangCai_2020} developed a multi-scale deep neural network method for computing oscillatory Stokes flows in complex
domains by using first order systems.
We have established an error estimate for the proposed VPVnet method, which guarantees the convergence of the method for a small value of the loss function and vice versa, i.e. when the approximate solution is close to the exact solution, the loss function is of small value.
It's also worth emphasizing that the VPVnet method is divergence-free and pressure-robust, while classical inf-sup stable mixed finite elements for the Stokes'
equations are not pressure-robust.
We use deep neural network comprised of ResNet blocks to approximate the solutions of the Stokes' equations, which is shown to have stronger predictive ability than
traditional feedforward neural networks.
The loss function is computed using Gauss-Legendre quadrature rules and the optimization algorithm is composed of two steps: an Adam optimizer with self-adaptive learning rates
followed by a limited-memory BroydenFletcherGoldfarbShanno algorithm with bound constraints (L-BFGS-B) \cite{BFGS_1989} to finetune the results, which improves and accelerates the optimization process.
Furthermore, the approximate and expressive ability of the VPVnet method is illustrated with non-smooth test cases {in both 2D and 3D.}

The paper is organized as follows. In Section \ref{SectionVPV}, we present the fi{r}st order velocity-pressure-vorticity formulation for the Stokes' problem \eqref{EQ:Stokes} and its relative least square functionals. In Section \ref{SectionDLS}, we describe the structure of the VPV neural networks for approximating the solutions of the Stokes' problem. The convergence and error estimates of the method is established in Section \ref{Section:convergence}. In Section \ref{Section-numerical-experiments}, several numerical results will be presented to demonstrate the efficiency and accuracy of the method. Finally, conclusions and discussions are presented in Section \ref{SectionDisCon}.

Throughout this paper, we use the standard notation and definition for the Sobolev spaces $H^s(\Omega)$ and $H^s(\Gamma)$ with inner products and norms denoted by $(\cdot,\cdot)_{s,\Omega}$
and $(\cdot,\cdot)_{s,\Gamma}$ and $\|\cdot\|_{s,\Omega}$ and $\|\cdot\|_{s,\Gamma}$, respectively. When there is no confusion, $\Omega$ (and $\Gamma$) is often omit in the indices for simplicity.
$H^s_0(\Omega)$ is defined as
\begin{eqnarray*}
H^s_0(\Omega)=\{v\in H^s(\Omega): ~v|_{\partial\Omega}=0\},
\end{eqnarray*}
and  $L_0^2(\Omega)$ denotes the subspace of square integrable functions with zero mean.
We also set $\tilde{H}^s(\Omega) = H^s(\Omega)\cap L^2_0(\Omega)$.
And by $(\cdot,\cdot)_{(s_1, \cdots, s_n)}$ and $\|\cdot\|_{(s_1, \cdots, s_n)}$ we denote inner products and norms, respectively,
on the product space $H^{s_1}\times \cdots \times H^{s_n}$; when all $s_i$ are equal, we shall simply write $(\cdot,\cdot)_{s, \Omega}$ and $\|\cdot\|_{s, \Omega}$.
\section{Velocity-{p}ressure-{v}orticity formulation of Stokes' equations and least square functionals}\label{SectionVPV}
For the {sake} of clarity, we consider the VPV formulation in 2D in this section.
{The} {3D formulation} {can be found} in \cite{Gunzburger_1994,Bo-Nan_1994} and will not be detailed here, {while} {numerical experiments will be conducted in both 2D and 3D in Section \ref{Section-numerical-experiments}}.
By introducing the vorticity $w = \nabla \times \bm{u}$ as an auxiliary variable, and using the identity $\nabla\times(\nabla\times\bm{u})=-\Delta \bm{u} +\nabla (\nabla\cdot \bm{u})$, in addition with the incompressibility constraint $\nabla\cdot\bm{u}=0$,
the Stokes equation \eqref{EQ:Stokes} can be written as
\begin{equation}\label{EQ:Velocity-Pressure-Vorticity_formulation}
\left \{\begin{split}
 \nu\nabla \times w +\nabla p &= \bm{f},\qquad {\rm in}\  \Omega,\\
\nu( w-\nabla \times \bm{u}) &=0, \qquad {\rm in}\ \Omega,\\
\nabla \cdot \bm{u}&=0, \qquad {\rm in}\ \Omega,\\
\end{split}\right.
\end{equation}
or
\begin{equation}\label{EQ:Velocity-Pressure-Vorticity_formulation_2d}
\left \{\begin{split}
 p_x + \nu w_y - f^{(1)} &=0,\qquad {\rm in}\  \Omega,\\
 p_y - \nu w_x - f^{(2)} &=0,\qquad {\rm in}\  \Omega,\\
 \nu(w + u_y - v_x )&=0,\qquad {\rm in}\  \Omega,\\
 u_x + v_y &=0,\qquad {\rm in}\  \Omega,\\
\end{split}\right.
\end{equation}
with a velocity boundary condition
\begin{equation}\label{EQ:VPV_Vbc}
\bm{u} = \bm{g} \text{ on } \Gamma,
\end{equation}
and a constraint condition on the pressure
\begin{equation}\label{EQ:VPV_pbc}
\int_{\Omega}p d\Omega = 0.
\end{equation}
We have the following proposition \cite{Gunzburger_1994}:
\begin{proposition}
The boundary value problem \eqref{EQ:Velocity-Pressure-Vorticity_formulation} (or \eqref{EQ:Velocity-Pressure-Vorticity_formulation_2d}), \eqref{EQ:VPV_Vbc} and \eqref{EQ:VPV_pbc} is equivalent to the Stokes' problem \eqref{EQ:Stokes} in primitive variable form.
The problem has a unique solution for all smooth data $f^{(1)}$, $f^{(2)}$ and $\bm{g}$.
\end{proposition}

Besides the boundary condition \eqref{EQ:VPV_Vbc}, various boundary conditions are possible for system \eqref{EQ:Velocity-Pressure-Vorticity_formulation},
and the solvability of the boundary-value problem depends on the combination of the boundary conditions \cite{Gunzburger_1994}.
More generally, let $U =(\bm{u},w, p)$, where $\bm{u}=(u,v)^T$, the VPV formulation of the Stokes' equations can be written in general form of a {first-order} system
\begin{eqnarray}
\mathscr{L} U &= &
\begin{pmatrix}
  0 & 0 & ~~\nu \frac{\partial}{\partial y} & \frac{\partial}{\partial x}\\
  0 & 0 & -\nu \frac{\partial}{\partial x} & \frac{\partial}{\partial y}\\
\nu \frac{\partial}{\partial y} & -\nu\frac{\partial}{\partial x} & ~~\nu~~~~ & 0\\
     ~~\frac{\partial}{\partial x} &   ~~\frac{\partial}{\partial y} & 0    & 0\\
\end{pmatrix}
\begin{pmatrix}  u \\v \\ w \\ p \\\end{pmatrix}
=
\begin{pmatrix}{f}^{(1)}\\ {f}^{(2)}\\ 0 \\ 0 \end{pmatrix}=F, \text{ in }\Omega, \label{Eq:Linear_system}\\
\mathscr{R} U  &= & G, \text{ on } \Gamma, \label{Eq:Linear_system_bdc}
\end{eqnarray}
where $\mathscr{L}$ is an elliptic operator in the sense of Douglis and
Nirenberg \cite{AgmonDouglisNirenberg_1964}, and $\mathscr{R}$ {represents} the boundary operators.
Following the study in \cite{Gunzburger_1994}, we consider two choices for the boundary operators.
The first one is the boundary condition \eqref{EQ:VPV_Vbc} that imposes the velocity on the boundary, which can be rewritten in the following form
\begin{eqnarray}
\mathscr{R}_1 U =
\begin{pmatrix}
u^0\\
v^0
\end{pmatrix}
\text{ on } \Gamma, \label{Eq:Linear_system_bdc1}
\end{eqnarray}
where $\bm{u}^0$ is a given function defined on $\Gamma$.
The second boundary operator impose{s} the pressure and the normal component of velocity,
\begin{eqnarray}
\mathscr{R}_1 U =
\begin{pmatrix}
\bm{u}_n^0\\
p^0
\end{pmatrix}
\text{ on } \Gamma, \label{Eq:Linear_system_bdc2}
\end{eqnarray}
where $\bm{u}_n$ represents the normal component of velocity, $\bm{u}_n^0$ and $p^0$ are given functions defined on $\Gamma$.
Recall that one also has to impose the additional constraint on the pressure \eqref{EQ:VPV_pbc} to guarantee the uniqueness of solutions.
\subsection{Least square functional}
We consider the least square functional for the VPV system \eqref{Eq:Linear_system}-\eqref{Eq:Linear_system_bdc} that is defined in terms of the norms indicated by the Agmon-Douglis-Nirenberg theory \cite{AgmonDouglisNirenberg_1964}.
The {reader} is referred to \cite{Bo-Nan_1990,Gunzburger_1994,BochevGunzburger_1998,DeangGunzburger_1998} for more details concerning its derivation, only the main results are recalled here.
Let $\mathscr{L}$ be characterized by two systems of integers $\{s_i\}= \{s_1, s_2, s_3, s_4\}$, $s_i\leq0$ and $\{t_j\}= \{t_1, t_2, t_3, t_4\}$, $t_j\geq0$, which are attached to the equations and the unknowns, respectively.
{Here variable $s_i$ corresponds} to the $i$th equation and $t_j$ to the $j$th dependent variable.
The least square functional for system \eqref{Eq:Linear_system}-\eqref{Eq:Linear_system_bdc} is defined {as}
\begin{eqnarray}
J(U)&=&\|\mathscr{L} U -F \|^2_{(-s_1,-s_2,-s_3,-s_4)} \label{Eq:Lsfunctional}\\
    &=&\|\nu\nabla \times w +\nabla p - \bm{f}\|^2_{(-s_1,-s_2)}  + \|\nu(w - \nabla \times \bm{u})\|^2_{-s_3} + \|\nabla \cdot\bm{u}\|^2_{-s_4} \nonumber,
\end{eqnarray}
where $\{s_i\} $ and $\{t_j\}$ are indices for which the operators $\mathscr{L}$ and $\mathscr{R}$ satisfy the ellipticity supplementary and complementing conditions  \cite{ChangBo-Nan_1990,Gunzburger_1994} and that the VPV system has a unique solution.
The minimization of \eqref{Eq:Lsfunctional} is {well-posed} over a suitable subspace $\bm{U}$ of $H^{t_1}(\Omega)\times {H}^{t_2}(\Omega)\times H^{t_3}(\Omega)\times H^{t_4}(\Omega)$.
The values of $\{s_i\}$ and $\{t_j\}$ depend on particular boundary operators, which will be specified {later}.

The least squares principle is then given by: seek $U=(\bm{u}, w,p)\in \bm{U}$ such that\\
\begin{eqnarray}\label{Eq:minimization}
J(U)\leq J(\uwave{U}),\quad \forall \ \uwave{U}=(\uwave{\bm{u}},\uwave{w},\uwave{p})\in \bm{U}.
\end{eqnarray}
And we have the following results{.}
\begin{proposition}\label{Prop:estimate}
The problem \eqref{Eq:minimization} has a unique minimizer $U\in \bm{U}$, and there exist{s} a constant $C>0$ such that
\begin{eqnarray}\label{Eq:estimation}
\|\bm{u}\|^2_{(t_1,t_2)} + \|w\|^2_{t_3} + \|p\|^2_{t_4}\leq C \|\bm{f}\|^2_{(-s_1,-s_2)}.
\end{eqnarray}
where $\{s_i\} $ and $\{t_j\}$ are indices such that the operators $\mathscr{L}$ and $\mathscr{R}$ satisfy the ellipticity supplementary and complementing conditions and that the VPV system has a unique solution \cite{ChangBo-Nan_1990,Gunzburger_1994}.
\end{proposition}

Then we specify the above results to different boundary operators.
For \textit{homogenous pressure normal velocity boundary condition \eqref{Eq:Linear_system_bdc2}}, we can choose the indices $s_i=0$ and $t_j =1$, then the least square functional involves only $L^2(\Omega)$-norms of the residuals of all the equations, i.e. we have that
\begin{eqnarray}
J(U)&=&\|\mathscr{L} U -F \|^2_{0} \label{Eq:Lsfunctional_pnvbdc}\\
    &=&\|\nu\nabla \times w +\nabla p - \bm{f}\|^2_{0}  + \|\nu(w - \nabla \times \bm{u})\|^2_{0} + \|\nabla \cdot\bm{u}\|^2_{0} \nonumber,
\end{eqnarray}
and $\bm{U} =\bm{H}_n^{1}(\Omega)\times H^{1}(\Omega)\times {H}_0^{1}(\Omega) $, where $\bm{H}_n^{1}(\Omega)$ denotes the subspace of $ H^{1}(\Omega)\times {H}^{1}(\Omega)$
whose members have normal components equal to zero on the boundary.

For the case when system \eqref{Eq:Linear_system} is supplemented with the \textit{homogenous velocity boundary condition \eqref{Eq:Linear_system_bdc1}},
we have $s_1=s_2=0$, $s_3=s_4=-1$ and $t_1=t_2=2$, $t_3=t_4=1$, thus the least square functional now involves $H^1(\Omega)$-norms of the residuals of some of the equations,
we have
\begin{eqnarray}
J(U)&=&\|\mathscr{L} U -F \|^2_{(0,0,1,1)} \label{Eq:Lsfunctional_vbdc}\\
    &=&\|\nu\nabla \times w +\nabla p - \bm{f}\|^2_{0}  + \|\nu(w - \nabla \times \bm{u})\|^2_{1} + \|\nabla \cdot\bm{u}\|^2_{1} \nonumber,
\end{eqnarray}
and $\bm{U} =(H^{2}(\Omega)\cap H_0^{1}(\Omega))^2\times H^{1}(\Omega)\times \tilde{H}^{1}(\Omega)$.
This implies that the velocity field should be approximated in finite-dimensional subspace of $H^2(\Omega)$, i.e. continuous differentiable functions.
Hence the straightforward application of the LS functional \eqref{Eq:Lsfunctional_vbdc} bears no direct advantage over least square principle based on the primitive variable stokes equations.
In order to use merely continuous functions,
we shall consider a partition-dependent functional \cite{Gunzburger_1994} which will involve only weighted $L^2$-norms of the residuals,
so the least square functional becomes
\begin{eqnarray}
J(U)&=&\|\nu\nabla \times w +\nabla p - \bm{f}\|^2_{0}
     + h^{-2}\|\nu(w - \nabla \times \bm{u})\|^2_{0} + h^{-2}\|\nabla \cdot\bm{u}\|^2_{0},
\end{eqnarray}
where $h$ is the meshsize of the elements of a regular partition of the domain $\Omega$, which will be detailed later in Section \ref{subsection-NCLF}.
We can now choose to approximate each unknown in a subspace of $H^1$, i.e {$\bm{U}= [H_0^{1}(\Omega)]^2\times H^{1}(\Omega)\times\tilde{H}^{1}(\Omega)$}. 
\section{Deep least {s}quare {n}eural {n}etwork method}\label{SectionDLS}
This section describes the {VPVnet} method for the Stokes' equations.
Let $(\bm{u}_{\bm{\theta}}, w_{\bm{\theta}}, p_{\bm{\theta}})$ be an approximation of the solutions of the {first-order} velocity-pressure-vorticity formulation \eqref{Eq:Linear_system} of the Stokes' equations obtained with a neural networks $U_{\bm{\theta}}$.
We hope to learn the parameters ${\bm{\theta}}$ of the neural network such that $(\bm{u}_{\bm{\theta}}, w_{\bm{\theta}}, p_{\bm{\theta}})$ approximate the solutions $(\bm{u}, w, p)$ as well as possible.
To this end, we define a loss function based on the least square functional \eqref{Eq:Lsfunctional} of the problem, which will be presented in subsection \ref{subsectionLoss}.
The structure of the deep neural network  $U_{\bm{\theta}}$ is first described as follows:
\subsection{DNN comprised of ResNet blocks}\label{Section-constructionDNN}
The solutions of the Stokes' equations are approximated by a deep neural network comprised of ResNet blocks, which takes spatial coordinates as input and predicts the corresponding velocity, vorticity and pressure fields, i.e., $U_{\bm{\theta}}: (x,y)\mapsto({\bm{u}}_{\bm{\theta}},w_{\bm{\theta}},p_{\bm{\theta}})$.
The construction of the DNN contains mainly the following two steps:

\textbf{First}, we construct a neural network which takes  $\bm{x}=[x,y]$ as input and neural networks $\bm{Y}=[\psi_{\bm{\theta}}(x,y), w_{\bm{\theta}}(x,y), p_{\bm{\theta}}(x,y)]$ as output.
It is comprised of 4 ResNet blocks, each block consists of two linear transformations, two activation functions and a residual connection,
and both the inputs $s$ and the output $t$ of the block are vectors in $R^m$ (see Figure \ref{figResnet} a schematic illustration of the Resnet networks).
The $i$th block can be expressed as:
\begin{equation}\label{Resnet_block}
t=\f_i(s) =\sigma\circ(\bm{W}_{i,2}\cdot \sigma\circ(\bm{W}_{i,1}s + \bm{b}_{i,1})+ \bm{b}_{i,2}) + s
\end{equation}
where $\bm{W}_{i,1}, \bm{W}_{i,2}\in R^{m\times m}$, $\bm{b}_{i,1},\bm{b}_{i,2}\in R^{m}$ are parameters associated with the block.
$\sigma$ is the activation function.
The dimension of $s$ and $t$ must be equal in \eqref{Resnet_block}, if it is not the case,
a linear projection $\bm{W}_{s}$ by the shortcut connections can be performed to match the dimensions, i.e. :
\begin{equation}\label{Eq:shortcut}
t=\f_i(s) =\sigma\circ(\bm{W}_{i,2}\cdot \sigma\circ(\bm{W}_{i,1}s + \bm{b}_{i,1})+ \bm{b}_{i,2}) + \bm{W}_{s}s.
\end{equation}
It has been shown by experiments in \cite{He_Resnet_2016} that the identity mapping is sufficient for addressing the degradation problem
and is economical, and thus $\bm{W}_{s}$ is only used when matching dimensions.

\textbf{Second}, based on the neural networks $\bm{Y}=[\psi_{\bm{\theta}}(x,y), w_{\bm{\theta}}(x,y), p_{\bm{\theta}}(x,y)]$, we derive the neural networks for $u_{\bm{\theta}}(x,y)$, $v_{\bm{\theta}}(x,y)$ by using the {\it automatic differentiation}.
We define

\begin{equation}
u_{\bm{\theta}} :=\frac{\partial\psi_{\bm{\theta}}}{\partial y},\quad v_{\bm{\theta}} :=-\frac{\partial\psi_{\bm{\theta}}}{\partial x}.
\end{equation}
Hence we have $\frac{\partial u_{\bm{\theta}}}{ \partial x} + \frac{\partial v_{\bm{\theta}}}{ \partial y} = \frac{\partial^2\psi_{\bm{\theta}}}{\partial y \partial x}- \frac{\partial^2\psi_{\bm{\theta}} }{\partial x \partial y}=0$, here we define $u_{\bm{\theta}}$ and $v_{\bm{\theta}}$ in this way such that the conservation law (i.e. $\nabla\cdot \bm{u}_{\bm{\theta}}=0$) is satisfied directly by its definition.
The networks $u_{\bm{\theta}}$, $v_{\bm{\theta}}$ can be derived in a variety of ways, here we use automatic differentiation method \cite{BaydinPearlmutterRadulSiskind2017} available in Tensorflow during the programming, e.g. $u_{\bm{\theta}}=\texttt{tf.gradients}(\psi_{\bm{\theta}},y)[0]$.
Note that Tensorflow uses reverse mode automatic differentiation, also known as backpropagation, for its gradient operation.
\begin{figure}[H]
	\begin{center}
    \includegraphics [width=0.8\textwidth]{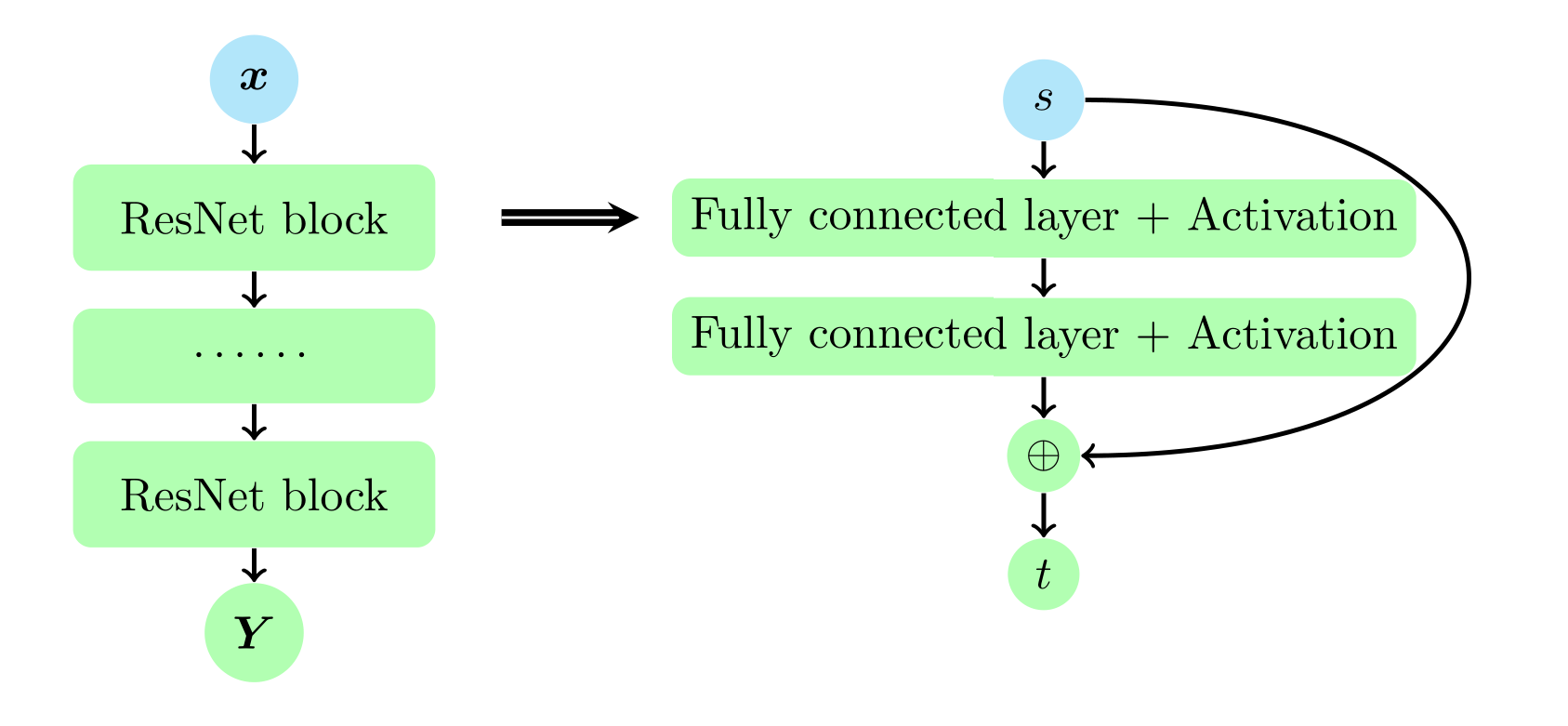}
	\end{center}
	\caption{A schematic illustration of the Resnet networks for the construction of $\f_{\bm{\theta}_2}$.}
	\label{figResnet}
\end{figure}

\begin{rremark}
The form of the ResNet block is flexible, which may contain two or three layers, or even more layers are possible.
But if the block has only a single layer, Eq \eqref{Resnet_block} is in fact a linear layer, for which no advantages have been observed in \cite{He_Resnet_2016}.
The parameters $\bm{W}$ and $\bm{b}$ can be randomly initialized using the Xavier scheme (which returns truncated normal distribution with the standard deviation of $1/m$).
As the choice of the activation function $\sigma$, $tanh$, $sinusoid$, $sigmoid$, etc. are possible, which will be verified and tested later in Section \ref{Section-numerical-experiments}.
Note that in \cite{Khodayi-MehrZavlano-VarNet2020}, the authors remarked that for variational type method $tanh$ might be incapable of capturing some solutions. 
\end{rremark}
\subsection{Loss function}\label{subsectionLoss}
Let $(\bm{u}_{\bm{\theta}}, w_{\bm{\theta}}, p_{\bm{\theta}})$ be an approximation of the solutions of system \eqref{Eq:Linear_system}- \eqref{Eq:Linear_system_bdc} obtained with the neural networks $U_{\bm{\theta}}(\bm{x})$.
We define a loss function based on the least square functional \eqref{Eq:Lsfunctional} which is written in the following form:
\begin{eqnarray}
\mathscr{J}(U_{\bm{\theta}})&=&\|\mathscr{L} U_{\bm{\theta}} -F \|^2_{(-s_1,-s_2,-s_3,-s_4)} +  \alpha\|\mathscr{R} U_{\bm{\theta}} -G\|^2_{1/2, \Gamma} \label{Eq:loss}
\end{eqnarray}
where we seek $U_{\bm{\theta}} \in H^{t_1}(\Omega)\times H^{t_2}(\Omega)\times H^{t_3}(\Omega)\times H^{t_4}(\Omega)$, $\{s_i\}$ and $\{t_i\}$ are admissible values for specific boundary operators, $\alpha$ is a positive constant, and the sobolev norm $\|\cdot\|_{1/2}$ is defined as
\begin{eqnarray}
&&\|\ell\|^2_{1/2,\Gamma}:= \int_{\Gamma}|\ell|^2 d\sigma + \int\int_{\Gamma\times\Gamma}\frac{|\ell(x) - \ell(y)|^2}{|x-y|^d} d\Gamma(x)d\Gamma(y). \label{demi-norm}
\end{eqnarray}
The treatment of boundary conditions is of great importance for numerical simulations.
For the sake of clarity, here we take into account the boundary conditions \eqref{Eq:Linear_system_bdc} directly in the least square functional.
By the trace theorem, the boundary norm is of $1/2$ order weaker than the interior norm, hence the norm $\|\cdot\|_{1/2}$ is employed.
Note that the extra penalty terms due to the boundary conditions could have negative effects on the training process and the final accuracy,
several efforts have been made recently  for the deep neural network methods to deal with boundary conditions.
For example, Liao et. al proposed a Nitche method \cite{Ming2019} to deal with the essential boundary conditions without significant extra computational cost.
Sheng et. al \cite{PFNNShengYang2020} developed a penalty free neural network methods, {where they} employ two networks, rather than just one, to construct the approximate solutions of the PDEs. With one networks satisfying the essential boundary conditions and the other handling the rest part of the domain, an unconstrained optimization problem is to be solved instead of a constrained one. We refer to the references above and herein \cite{LagarisLikas2000,McFallMahan_2009} for possible improve treatment of the boundary conditions.

Thus the VPVnet method is: find $U_{\bm{\theta}}: (x,y)\mapsto(\bm{u}_{\bm{\theta}},w_{\bm{\theta}},p_{\bm{\theta}})$ such that
\begin{eqnarray}\label{Eq:minimization_VPV}
\mathscr{J}(U_{\bm{\theta}}) = \min_{\bm{\hat{\theta}}\in R^n} \mathscr{J}(U_{\hat{\bm{\theta}}}).
\end{eqnarray}
And we have the following remark:
\begin{rremark}
Since $\nabla \cdot \bm{u}_{\bm{\theta}}$ vanish by construction of the neural network (see section \ref{Section-constructionDNN}), the neural network $U_{\bm{\theta}}$ can be in fact trained by minimizing the loss function:
\begin{eqnarray}\label{Eq:loss_reduction}
\mathscr{J}(U_{\bm{\theta}})&=&\|\nu\nabla \times w_{\bm{\theta}} +\nabla p_{\bm{\theta}} - \bm{f}\|^2_{(-s_1,-s_2)} \\
&+& \|\nu(w_{\bm{\theta}} - \nabla \times \bm{u}_{\bm{\theta}})\|^2_{-s_3} + \alpha\|\mathscr{R} U_{\bm{\theta}} -G\|^2_{1/2,\Gamma}.~~~~~~~~~~~~~~~\nonumber
\end{eqnarray}
More precisely, for the pressure normal velocity boundary condition \eqref{Eq:Linear_system_bdc2}, we have
\begin{eqnarray}\label{Eq:loss_pnv}
\mathscr{J}(U_{\bm{\theta}})&=&\|\nu\nabla \times w_{\bm{\theta}} +\nabla p_{\bm{\theta}} - \bm{f}\|^2_{0} \\
&+&\|\nu(w_{\bm{\theta}} - \nabla \times \bm{u}_{\bm{\theta}})\|^2_{0} + \alpha\|(\bm{u}_{n,{\bm{\theta}}},p_{\bm{\theta}})^T - (\bm{u}_n^0,p^0)^T \|^2_{1/2,\Gamma}.\nonumber
\end{eqnarray}
For the velocity boundary condition \eqref{Eq:Linear_system_bdc1}, i.e. \eqref{EQ:VPV_Vbc} , we have
\begin{eqnarray}\label{Eq:loss_vbdc}
\mathscr{J}(U_{\bm{\theta}})&=&\|\nu\nabla \times w_{\bm{\theta}} +\nabla p_{\bm{\theta}} - \bm{f}\|^2_{0} \\
&+& h^{-2}\|\nu(w_{\bm{\theta}} - \nabla \times \bm{u}_{\bm{\theta}})\|^2_{0} + \alpha\|\bm{u}_{\bm{\theta}} - \bm{g} \|^2_{1/2,\Gamma}.~~~~~~~~~~~~~~~\nonumber
\end{eqnarray}
\end{rremark}
\subsection{Numerical computation of loss function}\label{subsection-NCLF}
We consider {evaluating} the loss functions \eqref{Eq:loss_reduction} by numerical approximations,
i.e. we use Gauss-Legendre quadrature rules to compute the integration of the loss functions.
Once the loss function is computed, the gradient of the numerical approximation can be naturally calculated using automatic differentiation algorithm available in Tensorflow for the optimization process thereafter.
Similarly as in \cite{Caizhiqiang2020}, the computation make use of a partition of the domain.
Let $\D_h$ be an uniform square grid partition of the polygonal domain $\Omega\subset\mathbb{R}^2$ (See figure \ref{Fig:TD}).
Denote by $h\coloneqq \max_{{D} \in \D_h}h_{D}$ the meshsize of $\D_h$, where $h_{D}=\diam({D})$ is the diameter of the element ${D} \in \D_h$. Denote by $\E_h$ the edge set of $\D_h$, and $\E_h^0 \subset \E_h$ the set of all interior edges. The set of boundary edges is denoted as $\E_h^B\coloneqq\E_h \backslash\E_h^0$. The diameter of the edge $e\in \E_h$ is denoted as $h_e=\diam(e)$.
Throughout the paper, we employ one-point Gaussian quadrature, i.e. only the element center of each grid is considered as the quadrature point.
We use $\bm{x}_D$ and $\bm{x}_E$ to denote the quadrature points in element $D$ and on boundary $e$, respectively.
Since the Sobolev norm $\|\cdot\|_{1/2}$  \eqref{demi-norm} is not computationally feasible,
we approximate it by weighted $L^2$ norms with local weights $h_e^{-1/2}$.
By \eqref{Eq:loss_pnv},
we have the following discrete loss function
\begin{figure}[htbp]
\centering
\subfigure[QPs for boundary conditions ]{
\includegraphics [width=0.3\textwidth]{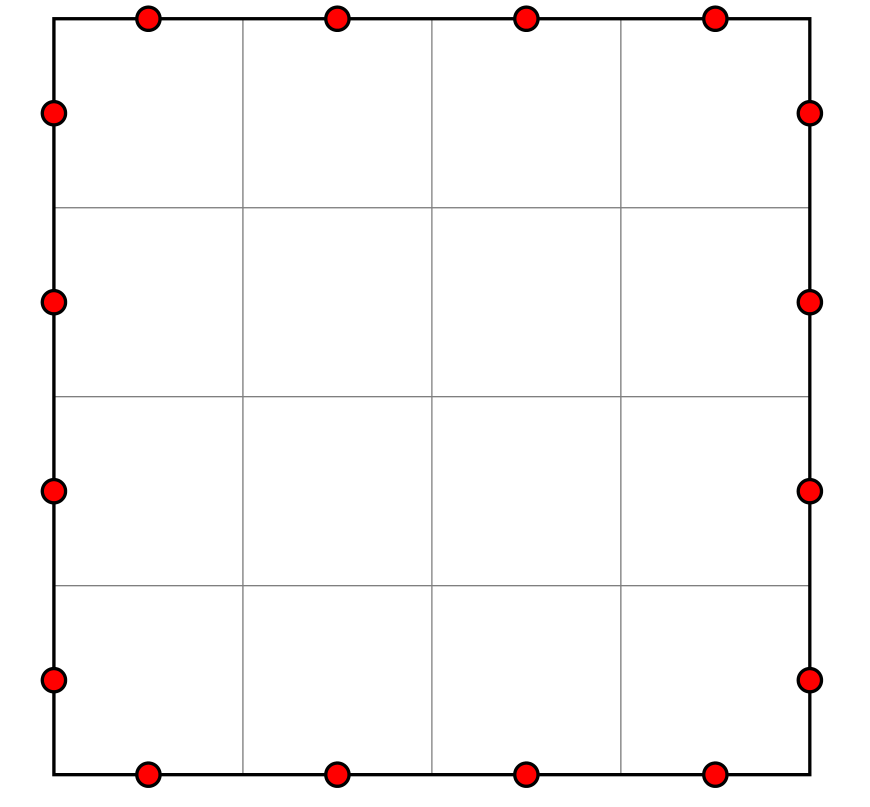}}
\hskip20pt
\subfigure[QPs for residuals of equations in $\Omega$ ]{
\includegraphics [width=0.3\textwidth]{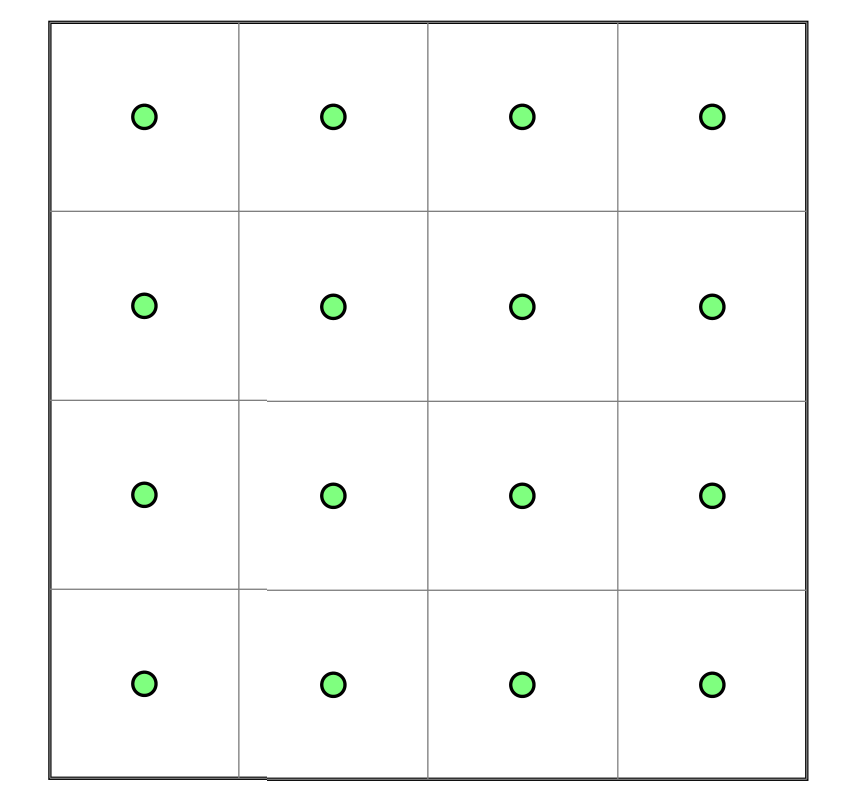}}
\caption{A uniform partition of domain $\Omega$ and related quadrature points (QPs) for the computation of the LS functional.}
\label{Fig:TD}
\end{figure}
\begin{eqnarray}\label{Eq:loss_discret}
&&\hat{\mathscr{J}}(\hat{U}_{\bm{\theta}}(\bm{x}))\\&\approx & \sum_{D \in \D_h} \left[(p_x + \nu w_y - f^{(1)})^2_{(\bm{x}_D,\bm{\theta})} + (p_y - \nu w_x - f^{(2)})^2_{(\bm{x}_D,\bm{\theta})} \right]|D|\nonumber\\
&+& \sum_{D \in \D_h} \nu^2h^{-2}(w + u_y - v_x)^2_{(\bm{x}_D,\bm{\theta})}|D| + \alpha\sum_{e \in \E^B_h}h_e^{-1}(\bm{u} -\bm{g}_D)^2_{(\bm{x}_E,\bm{\theta})}|e|,\nonumber
\end{eqnarray}
for the \textit{velocity boundary condition \eqref{Eq:Linear_system_bdc1}}, where $\alpha$ is positive constants, $|D|$ and $|e|$ are the measures of element $D$ and boundary $e$.
The discrete loss function for the \textit{pressure normal velocity boundary condition \eqref{Eq:Linear_system_bdc2}} can be derived similarly using \eqref{Eq:loss_vbdc} and will not be detailed here.
The discrete VPVnet method is to find $\hat{U}_{\bm{\theta}}: (x,y)\mapsto(\hat{\bm{u}}_{\bm{\theta}},\hat{w}_{\bm{\theta}},\hat{p}_{\bm{\theta}})$ such that
\begin{eqnarray}\label{Eq:minimization_dis}
\hat{\mathscr{J}}(\hat{U}_{\bm{\theta}}) = \min_{\bm{\hat{\theta}}\in R^n} \hat{\mathscr{J}}(\hat{U}_{\hat{\bm{\theta}}}).
\end{eqnarray}
The discrete loss function is lower bounded by zero, since this bound is attainable for the exact solution of the Stokes' equations, the values of the loss function reflects how well the functions $(\hat{\bm{u}}_{\bm{\theta}},\hat{w}_{\bm{\theta}},\hat{p}_{\bm{\theta}})$ approximate the solution of the Stokes' equations.
In practice, the problem \eqref{Eq:minimization_dis} is solved by an iterative method, such as the stochastic gradient decent method, which will be detailed in the section of numerical experiments.

\begin{rremark}
Note that here we use Gauss-Legendre quadrature rules relying on a partition of the domain
to compute the integration of the loss functions.
In this way, the numerical error of evaluating the integral can be computed to a desired accuracy
through either uniform or adaptive partition of the domain,
which can be easily estimated and controlled.
Otherwise, we can sample some points in the domain
and use those sampled points to define the discrete loss function,
then the method become a meshless method.
For example, we sample $M_i$ points inside $\Omega$, $M_{\partial\Omega}$  points on $\Gamma$, and define the discrete loss function by the following weighted mean squared error,
\begin{eqnarray}\label{Eq:loss_discret}
&&\hat{\mathscr{J}}(\hat{U}_{\bm{\theta}}(\bm{x}))\\&\approx & \sum_{M_i} \left[(p_x + \nu w_y - f^{(1)})^2_{(\bm{x}_{M_i},\bm{\theta})} + (p_y - \nu w_x - f^{(2)})^2_{(\bm{x}_{M_i},\bm{\theta})} \right]\nonumber\\
&+&\alpha_2\sum_{M_i} \nu^2(w + u_y - v_x)^2_{(\bm{x}_{M_i},\bm{\theta})} + \alpha_2\sum_{M_{\partial\Omega}}(\bm{u} -\bm{g}_{M_{\partial\Omega}})^2_{(\bm{x}_{M_{\partial\Omega}},\bm{\theta})},\nonumber
\end{eqnarray}
where $\alpha_2$ is a parameter to be determined and we adopt $\alpha_2 = 2500$ in the computation of test cased \eqref{Eq:testcase2D} and \eqref{Eq:testcase3D} on irregular domain (cf. Figure \ref{figmeshless:test1_test2}).
The derivative of the discrete loss function is calculated using automatic differentiation algorithm available in Tensorflow.
\end{rremark}
\section{Convergence analysis of VPVnet method}\label{Section:convergence}
According to the universal approximation property \cite{Hornik_1991}, the neural networks can approximate any continuous function with small errors.
Hence approximation errors of the neural network $U_{\bm{\theta}}$  can be small for any $U=(\bm{u},w,p)\in [C^0(\Omega)]^4$.
In this section, we will {analyze the} convergence of the VPVnet method.
The following is a relevant theorem \cite{XieCao_2011,Biswas_2020} that will be used in the demonstration:
\begin{theorem}\label{Th:approximation}
Suppose that $\sigma \in C^{\infty}(R)$, $\sigma^{(\nu)}(0)\neq 0$ for $\nu = 0,1,\cdots,$ and $K \subset R^d$ is any compact set.
If $f\in C^k(K )$, then a function $\varphi_n$ represented by a DNN with complexity $n\in N$ {exists} such that
\begin{eqnarray}
\|D^{\alpha}f- D^{\alpha}\varphi_n\|_{C(K)} = O\left(\frac{1}{n^{(k-|\alpha|)/d})}\omega(D^{\beta}f, \frac{1}{n^{1/d}}\right)
\end{eqnarray}
holds for all multi-indexes $\alpha$, $\beta$ with $|\alpha|\leq k$, $|\beta|=k$, where
\begin{eqnarray}
\omega(g,\delta) = \sup_{x,y \in K, |x-y|\leq\delta}|g(x) - g(y)|.
\end{eqnarray}
\end{theorem}

Recall that the VPV system \eqref{Eq:Linear_system}-\eqref{Eq:Linear_system_bdc} is a uniformly bounded elliptic system in the sense of Douglis and Nirenberg.
The system is well posed and has a unique solution.
Hence, the minimization problem \eqref{Eq:minimization_VPV} has a unique solution, with the value of the infimum being 0, and the infimum is attained at the solution $U$.
In this study, for simplicity, we consider only the case $s_i=0$ and $t_j =1$, and $G =0$, although the general case is similar.
According to proposition \ref{Prop:estimate}, we have the following a priori estimate:
\begin{eqnarray}\label{Eq:estimate_pnv}
\|\bm{u}\|^2_{1} + \|w\|^2_{1} + \|p\|^2_{1} \leq C\|\bm{f}\|^2_{0}.
\end{eqnarray}
First, we show that when the neural network $U_{\bm{\theta}}$ and the actual solution $U$ are close to each other, we can control the loss function \eqref{Eq:loss}.
\begin{theorem}\label{Th:convergence_1}
Let $U = (\bm{u}, w,p) \in [H^{1}(\Omega)]^4$ be the solution of system \eqref{Eq:Linear_system}-\eqref{Eq:Linear_system_bdc}, then there exist{s} $U_{\bm{\theta}}:(x,y)\mapsto(\bm{u}_{\bm{\theta}}, w_{\bm{\theta}},p_{\bm{\theta}})$ with $\|U - U_{\bm{\theta}} \|_1\leq \varepsilon$ such that
\begin{eqnarray}
J(U_{\bm{\theta}}) = \|\mathscr{L}U_{\bm{\theta}} -F \|^2_0 + \|\mathscr{R}U_{\bm{\theta}}\|^2_{1/2,\Gamma}\leq c\varepsilon^2 \text{ and } \| U_{\bm{\theta}}\|_1 \leq \tilde{M},
\end{eqnarray}
where $c$ is a positive constant.
\end{theorem}
\begin{proof}
We remark first that given any $\varepsilon>0$ , by Theorem \ref{Th:approximation},
there exist{s} a DNN $U_{\bm{\theta}}$ such that $\|U - U_{\bm{\theta}} \|_1\leq \varepsilon$.
It follows that
\begin{eqnarray}
\| \mathscr{L}U_{\bm{\theta}} - F\|^2_{0}
& = &\|\mathscr{L}U_{\bm{\theta}} - \mathscr{L}U\|^2_{0} \label{Eq:ineq_domain}\\
&\leq &C^2_{\mathscr{L}}\|U_{\bm{\theta}} - U\|^2_{1}\nonumber\\
&\leq &C^2_{\mathscr{L}}\varepsilon^2, \nonumber
\end{eqnarray}
where $C_{\mathscr{L}}$ is the operator norm bound of $\mathscr{L}$.
Since $ G =\bm{0}$, we also have
\begin{eqnarray}
\|\mathscr{R}U_{\bm{\theta}}\|^2_{1/2,\Gamma} & = &\|\mathscr{R}U_{\bm{\theta}} - G\|^2_{1/2,\Gamma} \label{Eq:ineq_bdc}\\
& = & \|\mathscr{R}U_{\bm{\theta}} - \mathscr{R}U\|^2_{1/2,\Gamma}\nonumber\\
& \leq & C^2_{\mathscr{R}}\|U_{\bm{\theta}} - U\|^2_{H^1(\Omega)} \leq  C^2_{\mathscr{R}}\varepsilon^2,  \nonumber
\end{eqnarray}
where $C_{\mathscr{R}}$ is the norm bound of the boundary operator.
Combining \eqref{Eq:ineq_domain} and \eqref{Eq:ineq_bdc}, we have
\begin{eqnarray}
J(U_{\bm{\theta}}) = \|\mathscr{L}(\w_{\bm{\theta}}) -F \|^2_0 + \|\mathscr{R}U_{\bm{\theta}}\|^2_{1/2,\Gamma} \leq c\varepsilon^2.
\end{eqnarray}
Moreover, by \eqref{Eq:estimate_pnv}, we have
\begin{eqnarray}
 \| U_{\bm{\theta}}\|_{1} & \leq &\|U_{\bm{\theta}} - U \|_{1} + \|U\|_{1}\\
 & \leq & \varepsilon + M = \tilde{M}, \nonumber
 \end{eqnarray}
 where $M$ is the norm bound of $\bm{f}$.
\end{proof}

Second, we show that by controlling the loss function \eqref{Eq:loss}, we can have an approximation $U_{\bm{\theta}}$ to the solution $U$ with small errors.
The error estimate is given in the theorem below.
\begin{theorem}\label{Th:convergence_2}
Let $U = ({\bm{u}}, w,p)\in [H^{1}(\Omega)]^4$ be the solution of system \eqref{Eq:Linear_system}-\eqref{Eq:Linear_system_bdc} and let $U_{\bm{\theta}}:(x,y)\mapsto({\bm{u}_{\bm{\theta}}}, w_{\bm{\theta}},p_{\bm{\theta}})$  be such that
\begin{eqnarray}
J(U_{\bm{\theta}}) = \|\mathscr{L}U_{\bm{\theta}} -F\|^2_{0} +  \|\mathscr{R}U_{\bm{\theta}}\|^2_{1/2,\Gamma}\leq \varepsilon^2. \label{Eq:ErrorEstimatesJ}
\end{eqnarray}
Then there exist{s} a constant $c>0$ such that
\begin{eqnarray}
\|U - U_{\bm{\theta}} \|_{0}\leq \|U - U_{\bm{\theta}} \|_{1}\leq c\varepsilon.
\end{eqnarray}
\end{theorem}
\begin{proof}
Denote $\mathscr{L} U_{\bm{\theta}} = F_{\varepsilon}$ and $\mathscr{R}U_{\bm{\theta}}= G_{\varepsilon}$.
From \eqref{Eq:ErrorEstimatesJ}, we have $\|F_{\varepsilon} - F\|_{0}\leq \varepsilon$ and $\|G_{\varepsilon}\|_{1/2, \Gamma}\leq \varepsilon$.
Consider the inverse boundary operator $\mathscr{R}^{-1}:[H^{1/2}(\Gamma)]^2\rightarrow [H^1(\Omega)]^4$, which is linear and bounded such that $ \mathscr{R}\mathscr{R}^{-1}= I$,
where $I$ is the identity operator.
Let ${\tilde{U}_{\bm{\theta}}} = U_{\bm{\theta}} - \mathscr{R}^{-1}G_{\varepsilon}$, we have
\begin{equation}\label{Eq:inverse}
\left \{\begin{split}
&\mathscr{L} \tilde{U}_{\bm{\theta}} = F_{\varepsilon} - \mathscr{L}\mathscr{R}^{-1}G_{\varepsilon} ,\quad {\rm in}\ \Omega,\\
&\mathscr{R} \tilde{U}_{\bm{\theta}} = 0,  \quad {\rm on}\ \Gamma.\\
\end{split}\right.
\end{equation}
Subtracting system \eqref{Eq:inverse} into system \eqref{Eq:Linear_system}-\eqref{Eq:Linear_system_bdc} leads to
\begin{equation}
\left \{\begin{split}
&\mathscr{L} \tilde{U}_{\bm{\theta}} - \mathscr{L}U = F_{\varepsilon} - \mathscr{L}\mathscr{R}^{-1}G_{\varepsilon} - F,\quad {\rm in}\ \Omega,\\
&\mathscr{R} \tilde{U}_{\bm{\theta}} - \mathscr{R}U = 0,  \quad {\rm on}\ \Gamma,\\
\end{split}\right.
\end{equation}
Note that $\mathscr{R}^{-1}\in [H^1(\Omega)]^4$, since $\mathscr{L}$ is a fourth order elliptic operator, we have $\mathscr{L}\mathscr{R}^{-1}G_{\varepsilon}\in [H^{-3}(\Omega)]^4$.
Then by the Lax-Milgram, we have
\begin{eqnarray}
\|\tilde{U}_{\bm{\theta}} -U\|_1^2\leq C^2_{\mathscr{L}^{-1}}\|(F_{\varepsilon}-F) - \mathscr{L}\mathscr{R}^{-1}G_{\varepsilon} \|_{-3}^2,
\end{eqnarray}
Therefore, we have
\begin{eqnarray}
\|U_{\bm{\theta}} -U\|_1 & = &\|U -(U_{\bm{\theta}} - \mathscr{R}^{-1}G_{\varepsilon} ) - \mathscr{R}^{-1}G_{\varepsilon}\|_1\\
 &\leq & \|U - \tilde{U}_{\bm{\theta}}\|_1 + \|\mathscr{R}^{-1}G_{\varepsilon}\|_1 \nonumber\\
 & \leq & C_{\mathscr{L}^{-1}}\|(F_{\varepsilon}-F) - \mathscr{L}\mathscr{R}^{-1}G_{\varepsilon} \|_{-3} +  C_{\mathscr{R}^{-1}}\|G_{\varepsilon}\|_{1/2,\Gamma}.\nonumber
\end{eqnarray}
Since
\begin{eqnarray}
\|F_{\varepsilon}-F\|_{-3} \leq \|F_{\varepsilon}-F\|_{0} \leq \varepsilon,
\end{eqnarray}
and
\begin{eqnarray}
\|\mathscr{L}\mathscr{R}^{-1}G_{\varepsilon}\|_{-3}\leq C_{\mathscr{L}}\|\mathscr{R}^{-1}G_{\varepsilon}\|_1\leq C_{\mathscr{L}}C_{\mathscr{R}^{-1}}\|G_{\varepsilon}\|_{1/2,\Gamma}.
\end{eqnarray}
Thus
\begin{eqnarray}
\|U - U_{\bm{\theta}} \|_{0}\leq \|U - U_{\bm{\theta}} \|_{1}\leq c\varepsilon,
\end{eqnarray}
\end{proof}

The numerical error of the VPVnet method mainly comes from three parts,
the approximation/expressive error caused by the deep neural network $U_{\bm{\theta}}$, the numerical error of evaluating the least square functional/neural network through numerical quadratures, and the algebraic error caused by the optimization algorithms.
As shown in Theorems \ref{Th:convergence_1} and \ref{Th:convergence_2}, we can obtain a neural network such that $\|U_{\bm{\theta}}-U\|\leq \varepsilon$ by minimizing the loss function \eqref{Eq:minimization_VPV}.
On the other side, the Gauss-Legendre quadrature is a classical method that {is} often used for approximating the value of integrals.
The numerical error of evaluating the least square functional can be computed to a desired accuracy
through either uniform or adaptive partition of the domain $\Omega$ and $\Gamma$ \cite{Caizhiqiang2020}.
Note that common methods of optimization such as stochastic gradient descent (SGD) has been employed in this paper.
Although the convergence of the optimization methods depends on many computation factors and is still a hot topic of research,
as we observed in section \ref{Section-numerical-experiments}, the optimization algorithm employed works well for the test cases considered.
\section{Numerical experiments}\label{Section-numerical-experiments}
The purpose of this section is two-fold. First, we shall verify the performance and accuracy of the VPVnet method on some synthetic test cases.
{Second}, we will demonstrate the {capability} of the VPVnet method to approximate non-smooth/singular
solutions of the Stokes' equations, {using} test cases with singular solutions on L-shape domains, and a {lid-driven} cavity problem will be considered.

All the experiments are replicated three times to reduce variability of random initialization of the method of gradient decent and the medians of three training results are reported.
The optimization algorithm is composed of two steps: we first use the Adam optimizer {\cite{Kingma_Adam_2014}} for e.g. 5,000 iterations with self-adaptive learning rates as proposed in {Tensorflow}, then apply the limited-memory BroydenFletcherGoldfarbShanno algorithm with bound constraints (L-BFGS-B) {\cite{BFGS_1989}} to finetune the results, {so that errors from inexact optimization are suppressed}. The training process of L-BFGS-B is terminated automatically based on the increment tolerance.
Note that although both the loss functions for the pressure normal velocity boundary condition \eqref{Eq:loss_pnv} and for the velocity boundary condition \eqref{Eq:loss_vbdc}
provide satisfying numerical approximations, only numerical results obtained with the loss function \eqref{Eq:loss_vbdc} (or \eqref{Eq:loss_discret} in discrete form) is presented in the following, as the velocity boundary condition is most commonly used in practice.
The results for the pressure normal velocity boundary condition will not be present here for simplicity.
\subsection{Smooth test case}\label{subsection:smooth_testcase}
Consider the following test case defined in domains $\Omega =(0,1)^d$, $d=2,3$ such that \eqref{Eq:testcase2D} and \eqref{Eq:testcase3D} are the exact solutions.
The parameter $\nu =1$ in these test cases.
The external source $\bm{f}$ and the Dirichlet boundary data $\bm{g}$ is chosen such that the above functions $\bm{u}$ and $p$ are solutions
to the system \eqref{EQ:Stokes}.

\begin{subequations}
\begin{eqnarray}
&&\bm{u}(x,y) =
\begin{pmatrix}
~~\sin(x)\sin(x)\cos(y)\sin(y)\\
 -\cos(x)\sin(x)\sin(y)\sin(y)\\
\end{pmatrix},
\quad p(x,y) = ~~\cos(x)\cos(y), \label{Eq:testcase2D}\\
&&\bm{u}(x,y,z) =
\begin{pmatrix}
    x + x^2 + xy + x^3y\\
    y + xy + y^2 + x^2y^2\\
    -2z-3xz-3yz-5x^2yz\\
\end{pmatrix},
\quad p(x,y,z) =  xyz + x^3y^3z -\frac{5}{32}. \label{Eq:testcase3D}
\end{eqnarray}
\end{subequations}

\begin{figure}[!h]
\centering
\subfigure[Speed]{
\label{Fig1.sub.1}
\includegraphics [width=0.31\textwidth]{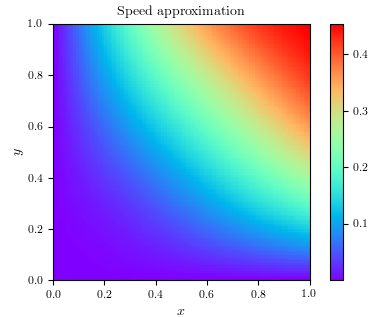}
}
\subfigure[Velocity ]{
\label{Fig1.sub.2}
\includegraphics [width=0.31\textwidth]{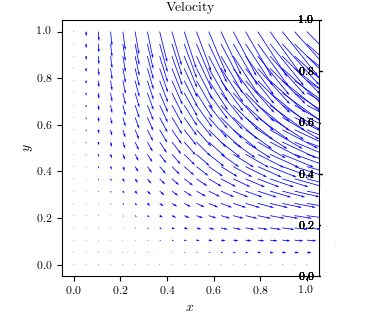}
}
\subfigure[Pressure ]{
\label{Fig1.sub.3}
\includegraphics [width=0.31\textwidth]{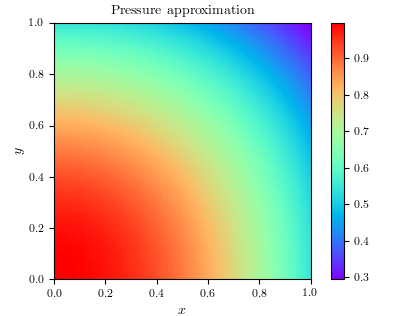}
}
\caption{Numerical results of test case \eqref{Eq:testcase2D} obtained with VPVnet method using $20\times20$ uniformly distributed quadrature points in $\Omega{=(0,1)^2}$.}
\label{fig:testcase 1}
\end{figure}

\begin{figure}[!h]
\centering
\subfigure[Speed]{
\label{Fig1.sub.1_3D}
\includegraphics [width=0.31\textwidth]{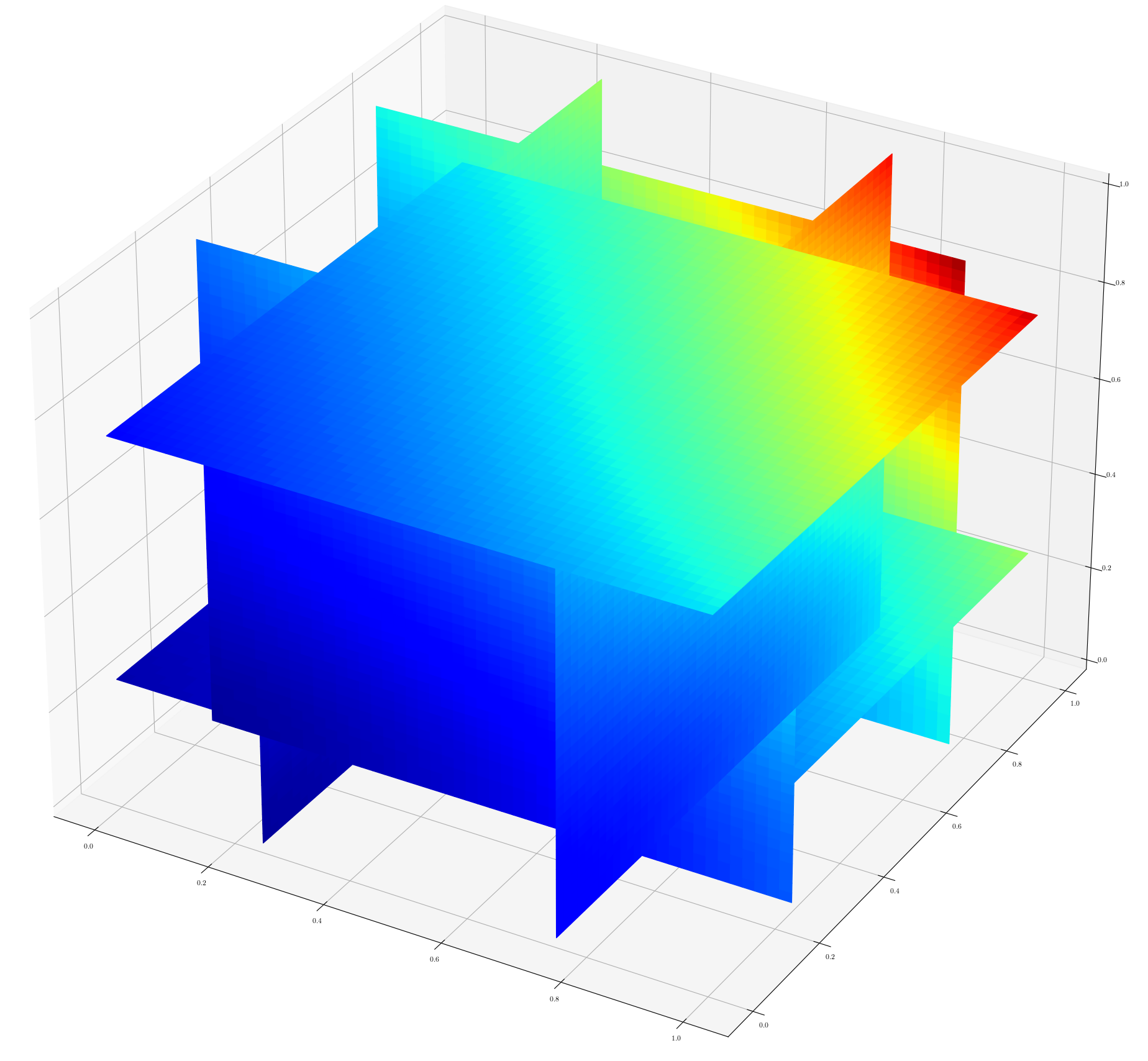}
}
\subfigure[Velocity ]{
\label{Fig1.sub.2_3D}
\includegraphics [width=0.31\textwidth]{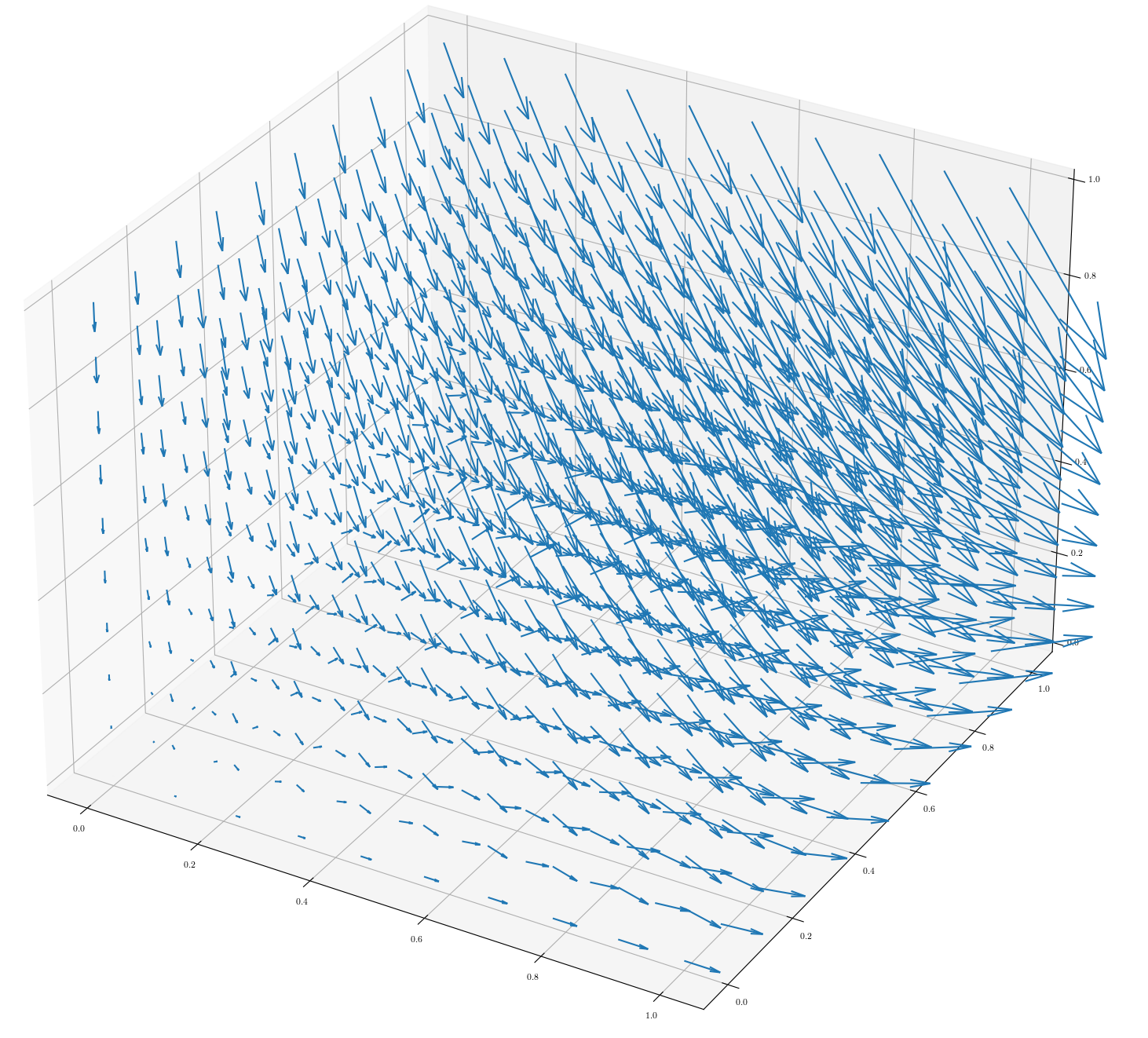}
}
\subfigure[Pressure ]{
\label{Fig1.sub.3_3D}
\includegraphics [width=0.31\textwidth]{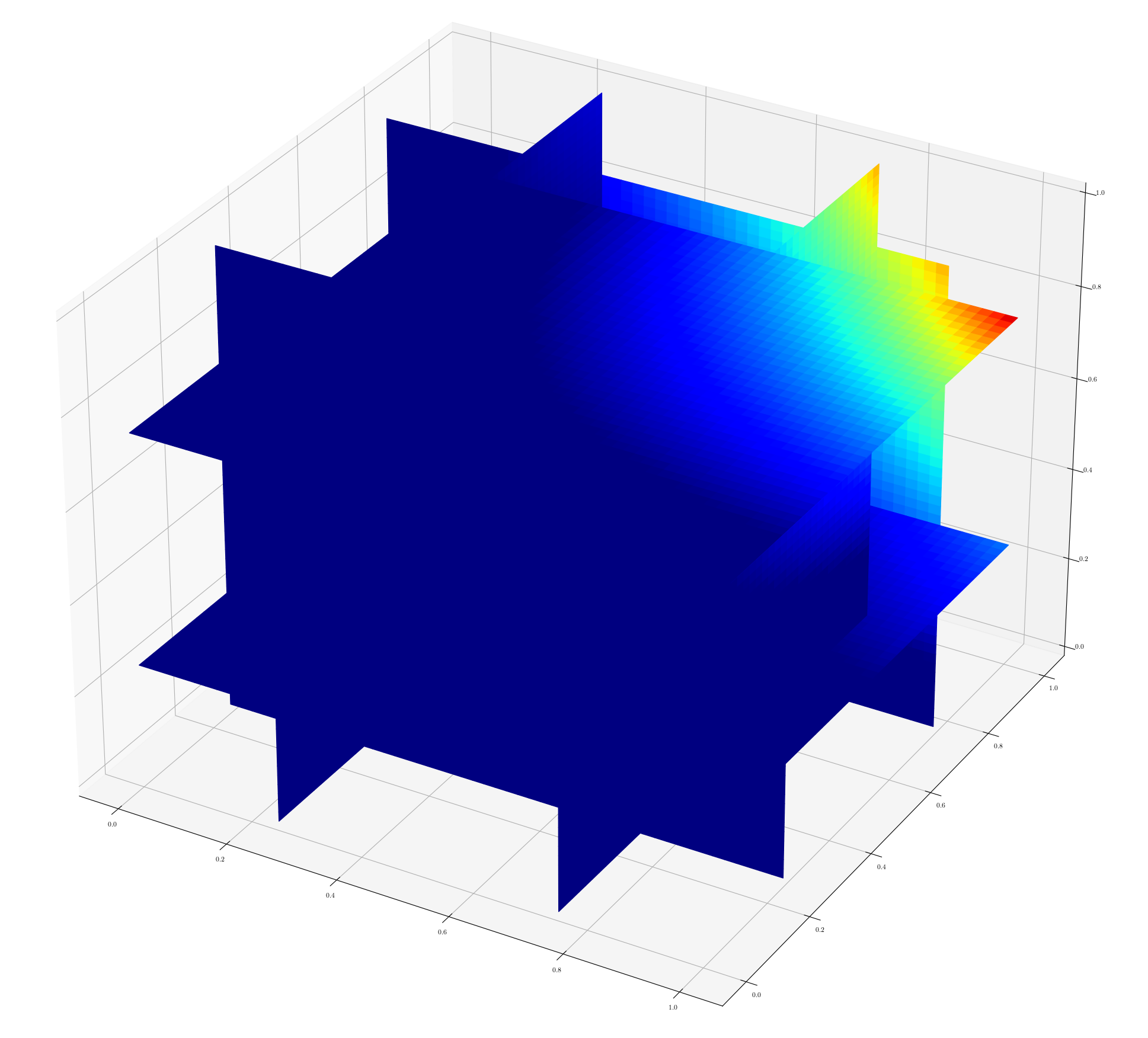}
}
\caption{Numerical results of test case \eqref{Eq:testcase3D} in 3D obtained with VPVnet method using $20\times20\times20$ uniformly distributed quadrature points in $\Omega{=(0,1)^3}$.}
\label{fig:testcase 1_3D}
\end{figure}

We first examine the accuracy of the VPVnet method {by varying} the number of quadrature points and the number of layers and neurons.
We use neural networks of different structures to approximate the solution of the Stokes' problem \eqref{Eq:testcase2D}.
For example, a neural network comprised of $2$ Resnet blocks, i.e. $4$ hidden layers, and each hidden layer admits $8$ neurons,
with total $267$ trainable parameters, is one of the neural networks considered.
For simplicity, we note the neural networks by its number of hidden layers $N_L$ and the number of neurons $N_N$ within each hidden layers, hence the above neural network is noted as a neural network of $4 \times 8$.
We also use neural networks of $8 \times 8$, $4\times16$, $8\times 16$ hidden layers and neurons, with total $555$, $915$, $2003$ trainable parameters, respectively, to approximates the solutions of the model problem \eqref{Eq:testcase2D}.

The discrete loss functions \eqref{Eq:loss_discret} is considered for the VPVnet method.
The quadrature points (QPs) for the evaluation of the loss function is uniform grid data on the domain, one point quadrature rules is considered, which consist of $50\times 50$ (or $20\times 20$) points in total.
Hence we have $h=1/50$ (or $h =1/20$).
And we choose $\alpha =1$.
QPs for the boundary condition of $u$, $v$ is uniformly distributed points on the boundary of $\Omega$ , which contains $50$ (or $20$) points on each edge, i.e $200$ (or $80$) points in total.
An illustration of these QPs is shown in Fig \ref{Fig:TD}.
As mentioned previously, the optimization problem is first computed with 2000 Adam iterations, and followed by at most 5000 L-BFGS-B iterations,
the learning rates are determined by existing self-adaption algorithms in Tensorflow.

\begin{figure}[!h]
\centering
\subfigure[Convergence of loss for test case 5.1(a)]{
\label{Fig5.sub.1}
\includegraphics [width=0.475\textwidth]{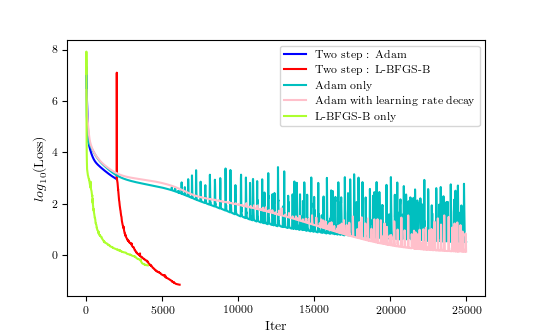}
}
\subfigure[Convergence of loss for test case 5.1(b)]{
\label{Fig5.sub.2}
\includegraphics [width=0.475\textwidth]{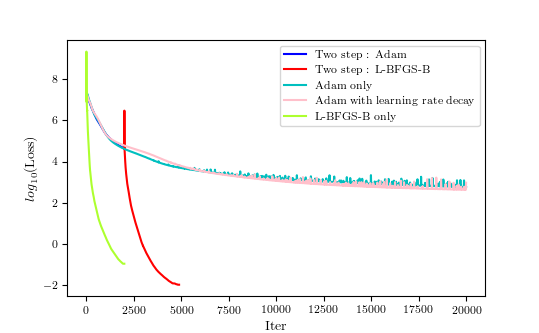}
}
\caption{Convergence of loss function over optimization
iteration with different optimization processes for test case 5.1(a) and 5.1(b), the curves are: $(1)$ Two step optimization : Adam part; $(2)$ Two step optimization : L-BFGS-B part;$(3)$ Adam optimizer only; $(4)$ Adam optimizer with a hyperparameter of learning rate $5e-4$ with exponential decay rate 0.98; $(5)$ L-BFGS-B optimizer only.}\label{fig5:test1_loss_convergence2}
\end{figure}

Figure \ref{fig5:test1_loss_convergence2} is the convergence of loss function over different optimization iterations.
As we can observe both in Fig. \ref{Fig5.sub.1} and \ref{Fig5.sub.2}, the two-step approach allow a much faster convergence and even a better precision than just using Adam only.
This is because BFGS uses the Hessian matrix (curvature in highly
dimensional space) to calculate the optimization direction and provides more accurate results.
However, if BFGS is used directly
without using the Adam optimizer, the optimization can rapidly converge to a local minimum \cite{Markidis_2021}. For this
reason, the Adam optimizer is used first to avoid local minima, and then the solution is refined by BFGS.
Here we use the L-BFGS-B method instead of the BFGS method in our work.

\begin{figure}[!h]
\centering
\includegraphics [width=0.5\textwidth]{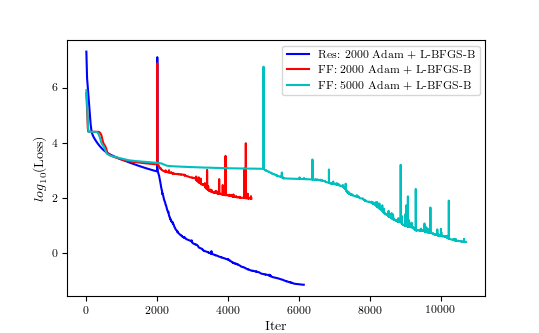}
\caption{Convergence of loss function over optimization
iterations obtained with different neural network structures for test case 5.1(a): $(1)$ neural network with Resnet blocks, 2000 Adam optimization iteration plus L-BFGS-B optimization process; $(2)$ feedforward neural network, 2000 Adam optimization iteration plus L-BFGS-B optimization process; $(3)$ feedforward neural network, 5000 Adam optimization iteration plus L-BFGS-B optimization iterations.}\label{fig:test1_loss_convergence_resnet}
\end{figure}

Figure \ref{fig:test1_loss_convergence_resnet} show the convergence of loss function over optimization
iteration obtained with two different neural network structures for test case 5.1(a), the first is neural network with Resnet (Res) blocks, the second is feedforward (FF) neural network, note that these two neural networks have the same number of 8 hidden layers and each hidden layers have the same number of 20 neurons, except that the neural network with Resnet blocks has a residual connection every two layers. We can see that the neural network with Resnet blocks converges faster and gives a better precision than classical feedforward network. Hence neural network with Resnet blocks is adopted for all the subsequent test cases in this work.

\begin{table}[!h]
\scriptsize
\begin{center}
\caption{Numerical errors of the VPVnet method for test case \eqref{Eq:testcase2D}.}\label{table1}
\begin{tabular}{|c|ccc|ccc|}
\hline
\tiny{Network}&\multicolumn{6}{|>{\columncolor{green!15}}c|}{ VPVnet method}\\
\cline{2-7}
\tiny{$N_{L}N_N$}&\multicolumn{3}{|c|}{ $20\times 20$ Quadrature points}&\multicolumn{3}{|c|}{ $50\times 50$ Quadrature points}\\
\cline{2-7}
 & $e_u$ & $e_v$ & $e_p$ & $e_u$ & $e_v$ & $e_p$  \\
\hline
$4\times 8$   &   7.88e-4 &  8.08e-4 &  1.36e-2  & 7.27e-4 &    7.58e-4 &  5.16e-3  \\
$8\times 8$   &   8.35e-4 &  8.26e-4 &  6.62e-3  & 5.09e-4 &    5.45e-4 &  3.31e-3  \\
$4\times 16$  &   6.58e-4 &  5.50e-4 &  1.02e-2  & 3.46e-4 &    4.47e-4 &  5.71e-3  \\
$8\times 16$  &   5.24e-4 &  5.49e-4&   4.56e-3  & 4.34e-4 &    4.99e-4 &  5.51e-3  \\
\hline
\end{tabular}
\end{center}
\end{table}

The numerical approximations obtained with the VPVnet method (using $sinusoid$ activation function) for the speed, velocity and pressure are plotted in Figure \ref{fig:testcase 1}.
The $L_2$ numerical errors of each variable $u,v,p$ are presented in Table \ref{table1}.
We observe that for this smooth test case, we can have a rather good approximation with a small number of hidden layers and neurons.
Although the numerical error varies due to the variability of random initialization of the parameters, we can observe that by using either more hidden layers, more number of neurons or more quadrature points, a tendency of better precision can be observed within certain limits.
\begin{table}[!h]
\scriptsize
\begin{center}
\caption{Numerical error of the VPVnet method for test case \eqref{Eq:testcase2D} using different activation function, a neural network of $8\times 16$ hidden layers and neurons, and $50\times 50$ quadrature points.}\label{table1.2}
\begin{tabular}{|c|ccc|}
\hline
\tiny{Activation}&\multicolumn{3}{|>{\columncolor{green!15}}c|}{ VPVnet method}\\
\cline{2-4}
\tiny{function} & $e_u$ & $e_v$ & $e_p$    \\
\hline
$Sin$      &   4.34e-4 &  4.99e-4  &  5.51e-3     \\
$Tanh$     &   5.11e-4 &  5.20e-4  &  9.98e-3     \\
$Sigmoid$  &   3.01e-3 &  3.36e-3  &  1.59e-2    \\
\hline
\end{tabular}
\end{center}
\end{table}

Table \ref{table1.2} shows the $L_2$ numerical errors using different activation functions.
The errors are obtained using a neural network with $8\times 16$ hidden layers and neurons, and $50\times 50$ quadrature points.
The results obtained with activation functions $Sin(x)$, $Tanh(x)$ and $Sigmoid(x)$ are all of satisfying precision.
We remark that the one of most commonly used activation function $ReLU$ is not considered here,
since the function is not differentiable everywhere,
while {first-order} derivative is presented in the loss functions and we use the automatic differentiation, i.e {\texttt{tf.gradient()}} to approximate the derivative,
which will cause problems.
A possible remedy to this problem might be using $max(0,x^3)$ instead of $ReLU:=max(0,x)$ as proposed in \cite{E_DRM_2018}.
Or we can approximate the derivative by using the finite difference quotient, as did by Cai $et$ $al.$ in \cite{Caizhiqiang2020}.
The {choice} of finite difference quotient for the approximation of the derivative of different variable in $2D$ or $3D$ is a question that needs careful consideration.
For the sake of clarity, we avoid such investigations here.
In the rest of this paper, without specification, we will use $Sin(x)$ as the activation function.

\begin{figure}[!h]
\centering
\subfigure[Sampling points, speed approximation and pressure approximation for 2D test case 5.1(a)]{
\label{Figmeshless.sub.1}
\includegraphics [width=0.28\textwidth]{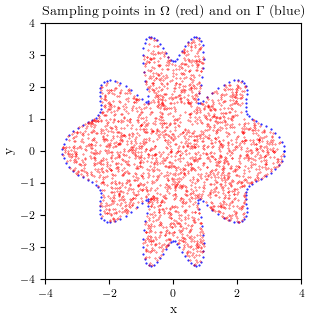}
\includegraphics [width=0.30\textwidth]{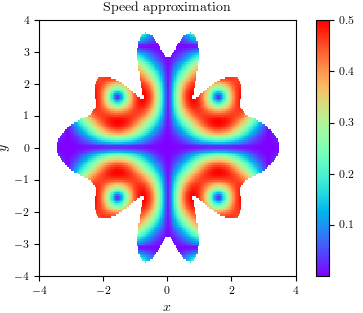}
\includegraphics [width=0.325\textwidth]{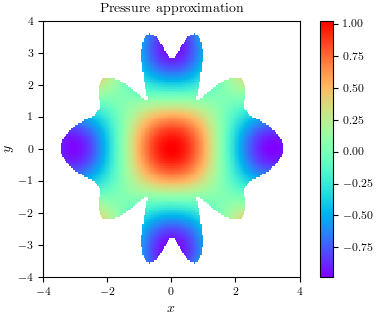}
}
\subfigure[Sampling points on the boundary (left) and inside the domain (right) for 3D test case 5.1(b)]{
\label{Figmeshless.sub.2}
\includegraphics [width=0.475\textwidth]{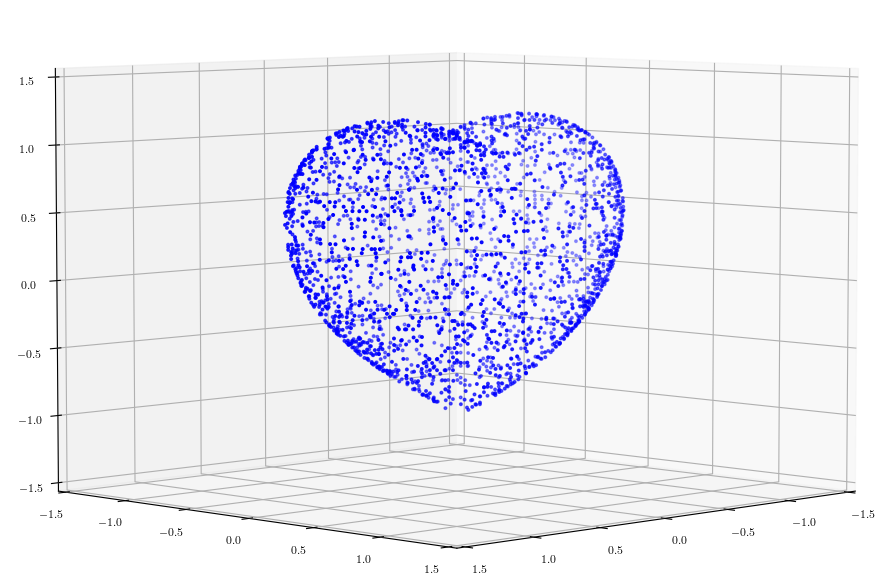}
\includegraphics [width=0.475\textwidth]{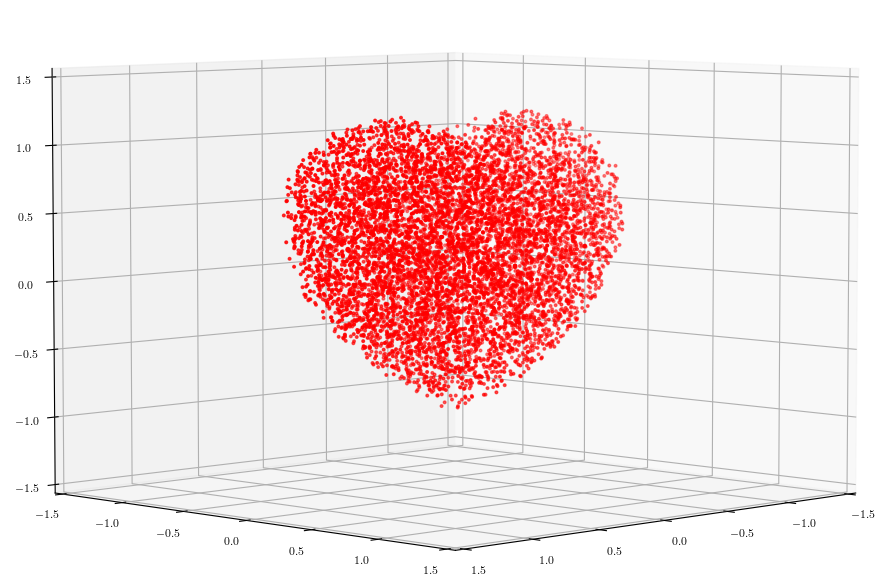}
}
\subfigure[Speed approximation and pressure approximation for 3D test case 5.1(b)]{
\label{Figmeshless.sub.3}
\includegraphics [width=0.475\textwidth]{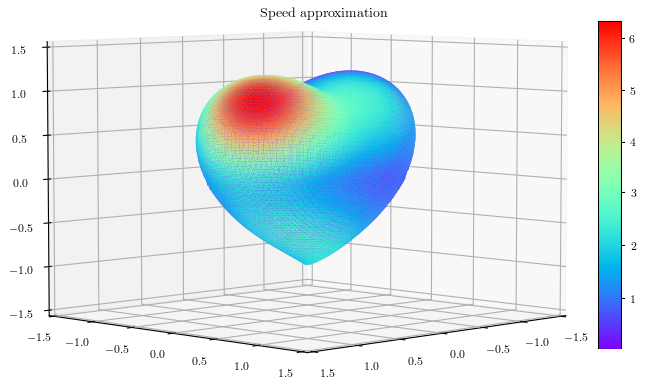}
\includegraphics [width=0.475\textwidth]{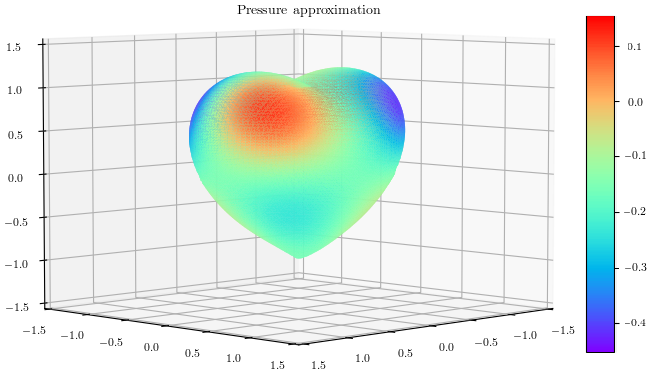}
}
\caption{Numerical results of  test case \eqref{Eq:testcase2D} and test case \eqref{Eq:testcase3D} obtained with VPVnet method using random sampling points in irregular domains}\label{figmeshless:test1_test2}
\end{figure}
We use neural networks of $8\times 16$ hidden layers and neurons to approximate the solution of the 3D test case \eqref{Eq:testcase3D}.
The quadrature points for the evaluation of the loss function is uniform grid data in the cube domain, which consist of $20\times 20\times 20$ points in total.
Note that BFGS instead of the L-BFGS-B algorithm is used to finetune the results obtained with the Adam optimizer in this 3D test case, as a better precision can be achieved by using the former.
The $L_2$ numerical errors obtained with the VPVnet method (using $sinusoid$ activation function) is $1.24e-5$, $1.37e-5$, $1.37e-5$, $4.85e-5$, $4.95e-3$, for variable $(\bm{u},p)$, where $\bm{u}=(u^{(1)}, u^{(2)}, u^{(3)})^T$ in 3D, respectively.
Similar precisions are reached for $tanh$ and $sigmoid$ activation functions, which will not be listed here for simplicity.
The numerical approximations obtained with the VPVnet method for the speed, velocity and pressure are plotted in Figure \ref{fig:testcase 1_3D}.
The numerical results show that VPVnet method is accurate and efficient in 3D.

The test cases \eqref{Eq:testcase2D} and \eqref{Eq:testcase3D} have been performed on more complex geometries using random sampling points.
In this case, the loss function \eqref{Eq:loss_discret} is employed.
The test case \eqref{Eq:testcase2D} is  preformed on a butterfly-like domain $\Omega$.
The boundary $\Gamma$ of the domain $\Omega$ is defined by the following equation \cite{TangLi_2017}:
\begin{eqnarray*}
\Gamma(x,y) =\left\{{(x,y)}^T\in R^2:\; x(\theta) = r(\theta)cos(\theta),\; y(\theta) = 1.4r(\theta)sin(\theta),\; \theta \in [0,2\pi]\right\},
\end{eqnarray*}
where $r(\theta) = 2\sqrt{cos^2(5\theta) + cos^2(2\theta) + cos^2(\theta)}$.
We used $200$ and $2500$ random sampling points (cf Fig. \ref{Figmeshless.sub.1}), respectively,  on the boundary and inside of the computation domain.
As for the 3D test case \eqref{Eq:testcase3D},
the domain boundary is described by the zero iso-value of the function $\phi$ \cite{HeMulin_2020}:
\begin{eqnarray*}
\Gamma(x,y,z) =\left\{{(x,y,z)}^T\in R^3:\; \phi(x,y,z) =\left(x^2 + \tfrac{9}{4}y^2 +z^2-1\right)^3 -x^2z^3 - \tfrac{9}{80}y^2z^3 =0\right\}.
\end{eqnarray*}
$2500$ and $8000$ random sampling points are employed (cf Fig. \ref{Figmeshless.sub.3}), respectively,  on the boundary and inside of the computation domain.
Convergence can be easily obtained for both test cases using the same optimization processes as previously used.
The numerical results (cf Fig. \ref{figmeshless:test1_test2}) show that the method is very flexible on complex geometries.
The $L_2$ numerical errors obtained with the VPVnet method (using $sinusoid$ activation function) is
$3.20e-03$, $3.16e-03$ and $3.34e-02$ for variable $(\bm{u},p)$, where $\bm{u}=(u^{(1)}, u^{(2)})^T$ in 2D, respectively (and $5.88e-05$, $1.30e-04$, $ 2.58e-05$, $7.77e-3$ for $(\bm{u},p)$ in 3D).
The precision can be further improved by adaptive sampling \cite{Wight_2020} which is the subject for future study.
\subsection{Pressure-{r}obust test case}\label{subsec:pressurerobust}
In this example, let $\Omega = (0, 1)^2$, and choose the
proper right-hand side function such that the following functions satisfy system \eqref{EQ:Stokes}.
\begin{eqnarray}\label{Eq:robust}
\bm{u}(x,y) =
\begin{pmatrix}
-e^x(y\cos(y)+\sin(y))\\
 ~~e^x y\sin(y)\\
\end{pmatrix},
\quad p(x,y) = 2e^x\sin(y).
\end{eqnarray}
Since $ -\Delta \bm{u} = (-2e^x\sin (y), - 2e^x\cos (y))^t $, it is easy to check
that when $\nu = 1$, we have a homogeneous equation with $\bm{f} = 0$.
When $\nu = 0$, $\bm{f}$ is a nonzero function.
It is well known that the classical finite element methods, such as the mini element, Taylor-Hood element, etc,
suffer from a common lack of robustness when $\nu \ll 1$.
The velocity error can become {extremely} large.
It is also called poor mass conservation sometimes, since for $H^1$-conforming mixed methods such large velocity errors are accompanied by large divergence errors.
In this test case, we will report the robustness results of the VPVnet method.

\begin{figure}[!h]
\centering
\subfigure[$\nu=1e-2$]{
\label{Fig2.sub.1}
\includegraphics [width=0.31\textwidth]{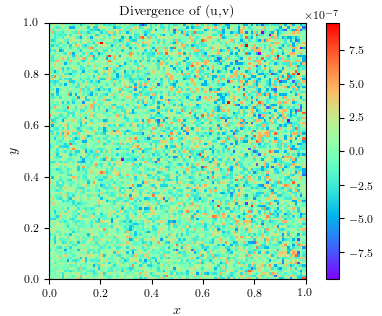}
}
\subfigure[$\nu=1e-4$]{
\label{Fig2.sub.2}
\includegraphics [width=0.31\textwidth]{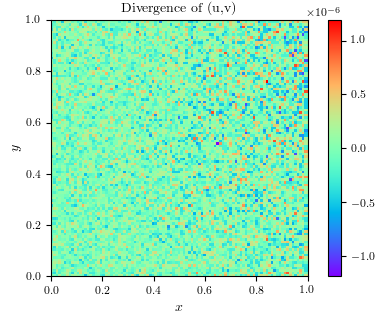}
}
\subfigure[$\nu=1e-6$]{
\label{Fig2.sub.3}
\includegraphics [width=0.31\textwidth]{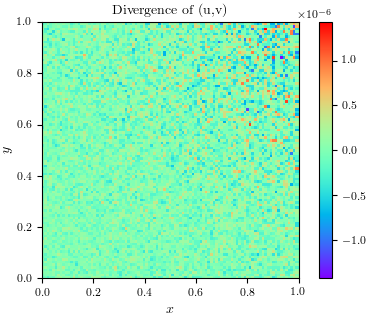}
}
\caption{Numerical results of divergence of $\bm{u}$ for test case \eqref{Eq:robust} obtained with VPVnet method using $50\times50$ uniformly distributed quadrature points in $\Omega$.}
\label{fig:testcase 2}
\end{figure}

\begin{table}[!h]
\scriptsize
\begin{center}
\caption{Numerical error of the VPVnet method with different value of $\nu$ for {test case} \eqref{Eq:robust}, using a neural network with $8\times 16$ hidden layers and neurons, and $20\times 20$ and $50\times 50$ quadrature points.}\label{table2}
\begin{tabular}{|c|ccc|ccc|}
\hline
&\multicolumn{6}{|>{\columncolor{green!15}}c|}{ VPVnet method}\\
\cline{2-7}
$\nu$&\multicolumn{3}{|c|}{ $20\times 20$ Quadrature points}&\multicolumn{3}{|c|}{ $50\times 50$ Quadrature points}\\
\cline{2-7}
 & $e_u$ & $e_v$ & $e_p$ & $e_u$ & $e_v$ & $e_p$  \\
\hline
$1e-2$   & 4.18e-4 &   1.35e-3  & 1.46e-3&   1.60e-4 &  5.86e-4   & 2.62e-3  \\
$1e-4$   & 2.06e-4 &   5.94e-4  & 1.89e-3&   2.26e-4 &  7.78e-4   & 2.85e-3  \\
$1e-6$   & 3.22e-4 &  9.23e-04  & 8.68e-4&   3.98e-4 &  1.18e-3   & 2.09e-3  \\
\hline
\end{tabular}
\end{center}
\end{table}
The numerical experiments are carried out for $\nu = 1e-2,\; 1e-4, \;1e-6$.
The numerical results are obtained using a neural network with $8\times 16$ hidden layers and neurons, and $20\times 20$ and $50\times 50$ quadrature points.
The optimization problem is computed with 2000 Adam and at most 5000 L-BFGS-B iterations, respectively.
The $L_2$ numerical errors obtained with the VPVnet method of variable $u,v,p$ are presented in Table \ref{table2}.
We obtain very good approximations of the solutions for any value of $\nu$, including the value $\nu$ very small.
The divergence of the velocity $div \bm{u} =u_x + v_y$ for different values of $\nu$ is plotted in Figure \ref{fig:testcase 2}.
Hence the VPVnet method is divergence-free and pressure-robust.
\subsection{Smooth test case on L-shape domain}
For further examination, the more complicated domain, L-shape domain, is considered to
validate in this research.
The domain is defined by $\Omega = (-1,1)^2 \backslash([0,1] \times [-1,0])$.
It is called the backward facing step domain of flow problems.
The parameter $\nu$ is of value $1$ for all domain, $\nu = 1$.
The test case is of exact solution as shown in Eq. \eqref{Eq:robust}.
The same test case has been considered in \cite{Vu-Huu_2020} and many other references.

\begin{figure}[!h]
\centering
\subfigure[Speed]{
\label{Fig3.sub.1}
\includegraphics [width=0.30\textwidth]{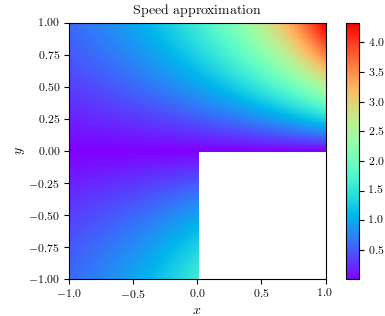}
}
\subfigure[Pressure]{
\label{Fig3.sub.2}
\includegraphics [width=0.30\textwidth]{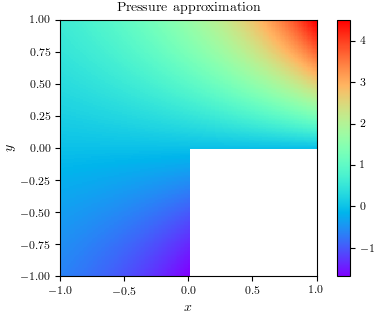}
}
\subfigure[Pressure plot in 3D ]{
\label{Fig3.sub.3}
\includegraphics [width=0.30\textwidth]{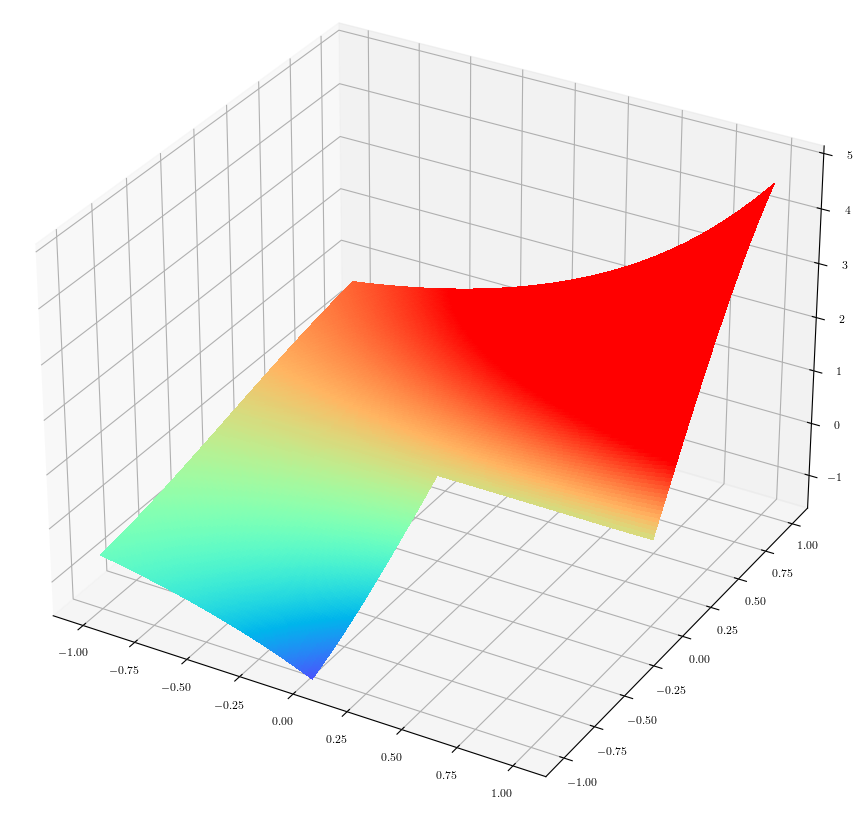}
}\\
\subfigure[Velocity ]{
\label{Fig3.sub.4}
\includegraphics [width=0.30\textwidth]{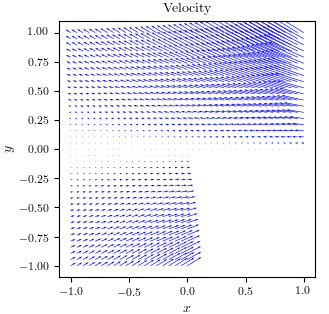}
}
\subfigure[Streamline ]{
\label{Fig3.sub.5}
\includegraphics [width=0.35\textwidth]{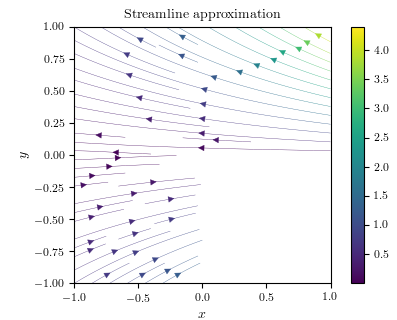}
}
\caption{Numerical results of test case \eqref{Eq:robust} on L-shape domain obtained with VPVnet method using nearly $50\times50\times3$ uniformly distributed quadrature points in $\Omega$.}
\label{fig:testcase 3}
\end{figure}
We obtain good approximations for the test case on L-shape domain with the VPVnet method.
Note that the test case has the same exact solution as \eqref{Eq:robust} in {subsection \ref{subsec:pressurerobust}}, hence similar computation configuration ha{s} been used and the numerical errors are not shown here again.
The numerical approximations obtained with the VPVnet method for the speed, velocity, pressure and the streamline are plotted in Figure \ref{fig:testcase 3}.
\subsection{Singular test case on L-shape domain}
{Hereafter}, we consider a singular {solution} on the L-shaped domain $\Omega:=(-1,1)^2$$\backslash$$[0,1)\times(-1,0]$.
The test case is proposed by Verf{\"u}rth in \cite{Verfurth_1996} and has been widely used in the literature \cite{ChernovMarcatiMascotto_2020,LiRuo_2019}.
Let $\omega=3\pi/2$, ${\delta} = 0.5444837$.
Given $(r, \theta)$ the polar coordinates at the re-entrant corner $(0,0)$, introduce the auxiliary function
\begin{equation}
\psi(r,\theta) = \frac{\sin((1+{\delta})\theta)\cos({\delta}\omega)}{1+{\delta}}-\cos((1+{\delta})\theta) - \frac{\sin((1-{\delta})\theta)\cos({\delta}\omega)}{1-{\delta}} + \cos((1-{\delta})\theta).
\end{equation}
The singular solution we approximate is
\begin{equation}\label{L-shape-singular}
\left \{\begin{split}
&u(x,y) =  r^{{\delta}}\left((1+{\delta})\sin(\theta)\psi(\theta) + \cos(\theta)\psi'(\theta)\right),\\
&v(x,y) =  r^{{\delta}}\left(\sin(\theta)\psi'(\theta)-(1+{\delta})\cos(\theta)\psi(\theta)\right),\\
&p(x,y) =  r^{{\delta}-1}\left((1+{\delta})^2\psi'(\theta) + \psi^{(3)}(\theta)\right)/(1-{\delta}).
\end{split}\right.
\end{equation}
The Stokes equation is homogeneous with this solution, i.e $\bm{f}=0$.
Moreover, the Dirichlet conditions are homogeneous along the edges except at the re-entrant corner. 
We emphasize that the pair $(\bm{u}, p)$ is analytical in $\Omega\backslash(0,0)$, but $\nabla \bm{u}$ and $p$ are singular at the origin. Especially, $\bm{u}\notin [H^2]^2(\Omega)$
and $p \notin H^1(\Omega)$. This solution reflects perfectly the typical behavior of the solution of the Stokes problem near a reentrant corner.
In general, for solution $(\bm{u},p)$ to the Stokes problem in a non-convex domain $\Omega$,
we can expect $\bm{u} \in [H^s(\Omega)]^2$ and $p \in H^{s-1}$ for $s<1+{\delta}$.
Hence the solution $\bm{u}$ is in $C^0$, and theoretically, the neural networks is capable of approximating a $C^0$ function at any error.
It is expected by construction that the VPVnet method can capture such singular solutions, which is confirmed by the following numerical results.

\begin{figure}[!htbp]
\centering
\subfigure[Exact(left) \& approximate(right) pressure for domain $(-1.0,1.0)^2$]{
\label{Fig4.sub.11}
\includegraphics [width=0.3250\textwidth]{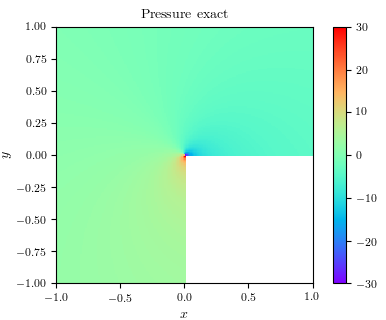}
\includegraphics [width=0.3250\textwidth]{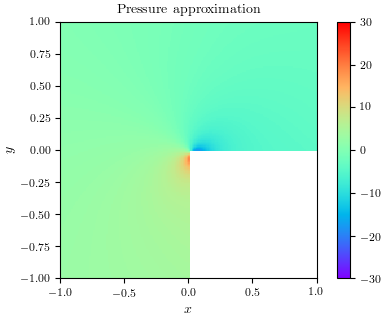}
}\\
\subfigure[Exact(left) \& approximate(right) pressure(zoom) domain $(-0.15,0.15)^2$]{
\label{Fig4.sub.12}
\includegraphics [width=0.3250\textwidth]{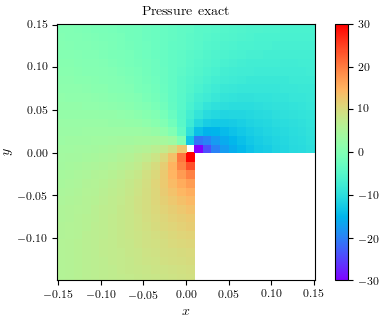}
\includegraphics [width=0.3250\textwidth]{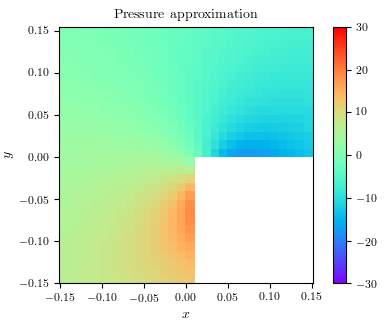}
}\\
\subfigure[Exact \& approximate pressure plot in 3D ]{
\label{Fig4.sub.2}
\includegraphics [width=0.3250\textwidth]{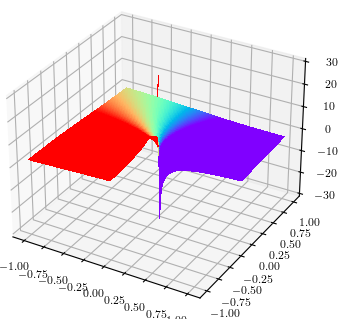}
\includegraphics [width=0.3250\textwidth]{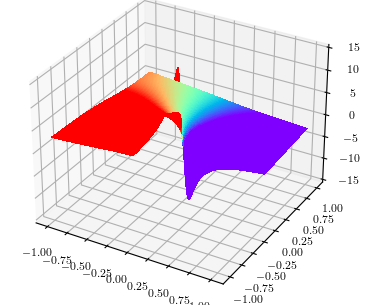}
}\\
\subfigure[Speed]{
\label{Fig4.sub.3}
\includegraphics [width=0.315\textwidth]{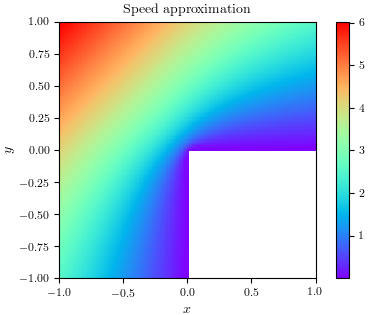}
}
\subfigure[Velocity ]{
\label{Fig4.sub.4}
\includegraphics [width=0.30\textwidth]{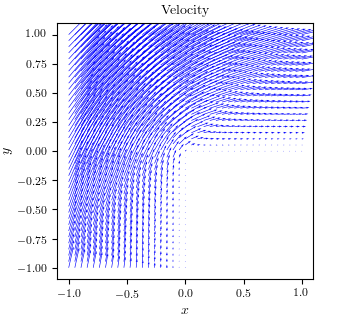}
}
\subfigure[Streamline ]{
\label{Fig4.sub.5}
\includegraphics [width=0.315\textwidth]{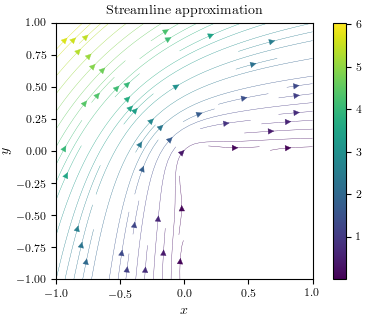}
}

\caption{Numerical results of test case \eqref{L-shape-singular} obtained with VPVnet method using $50\times50\times3$ quadrature points in $\Omega$.}
\label{fig:testcase 4}
\end{figure}

\begin{table}[!h]
\scriptsize
\begin{center}
\caption{Error of the VPVnet method for {test case} \eqref{L-shape-singular} using  $8\times 16$ hidden layers and neurons.}\label{table1.4}
\begin{tabular}{|c|ccc|}
\hline
\tiny{Quadrature}&\multicolumn{3}{|>{\columncolor{green!15}}c|}{ VPVnet method}\\
\cline{2-4}
\tiny{points} & $e_u$ & $e_v$ & $e_p$    \\
\hline
$20\times 20 \times 3$    &   2.75e-2  &  2.22e-2     &   5.21e-1  \\
$50\times 50 \times 3$    &   6.45e-3  &  5.92e-3     &   1.96e-1  \\
$100\times 100 \times 3$  &   8.60e-3  &  8.63e-3     &   3.17e-1  \\
\hline
\end{tabular}
\end{center}
\end{table}

The computation configuration is as follows:
The parameters $\alpha=1$ and a neural network with $12 \times 16$ hidden layers and
neurons is employed for all experiments in this test case.
Uniform partitions of size $20\times 20\times3$, $50\times 50\times3$ and $100\times 100\times3$  are considered.
(Note that we use $*\times*\times 3$ QPs because the L-shape domain is composed of 3 square domain.)
One point quadrature rules is employed for all cases.
Hence we have $h=1/20$ (or $h =1/50$, $h =1/100$) for the uniform partitions.
As for the boundary condition of $u$, $v$, the partition is in accordance with the partitions interior domain.
For the optimization, $2000$ Adam iteration is used, while more L-BFGS-B iterations is needed (less than 50000 iterations), respectively.

The numerical {solutions} obtained with the VPVnet method for the speed, velocity, pressure and the streamline, and the exact solution for the pressure are are plotted in Figure \ref{fig:testcase 4}.
The numerical results in Table \ref{table1.4} clearly show that the VPVnet method can provide good approximations for $\bm{u} \in [H^s(\Omega)]^2$, $2>s>1$.
While the errors obtained with the VPVnet method for $p \in H^{s-1}$, $2>s>1$ {are} much larger, due to the strong pressure singularities near the re-entrant corner $(0,0)$, which can also be observed in Figure \ref{Fig4.sub.11}, \ref{Fig4.sub.12} and \ref{Fig4.sub.2}.
The numerical errors can be reduced by refinement within a certain limit, then it will not get any better, {indicating} more efforts {are} needed to reach a satisfying precision.
Recall that only smooth activation function is considered here in this paper, and it has been shown in \cite{Caizhiqiang2020} that the piecewise activation functions, e.g. $Leaky$ $ReLU$ gives a better performance than the smooth activation functions for a interface problem,
hence it would be interesting to try continuous piecewise activation functions as remarked in subsection \ref{subsection:smooth_testcase}, which is subject to further investigations.
\subsection{Lid-driven cavity test cases}
The lid-driven cavity test case has been used as a benchmark problem for many numerical methods due to its
simple geometry and complicated flow behaviors \cite{Benchmarkliddriven1998}.
The computation domain is $\Omega=(0,1)^2$ in 2D.
And the Dirichlet boundary condition is given as $\bm{u}_{\Gamma} =\bm{g}=(1, 0)^T $ for $y = 1$ and
$\bm{g} = (0, 0)^T$ on the rest of the boundary.
\begin{figure}[!h]
\centering
\subfigure[Quadrature points in 2D]{\label{QPs_LDC_2D}
\includegraphics [width=0.35\textwidth]{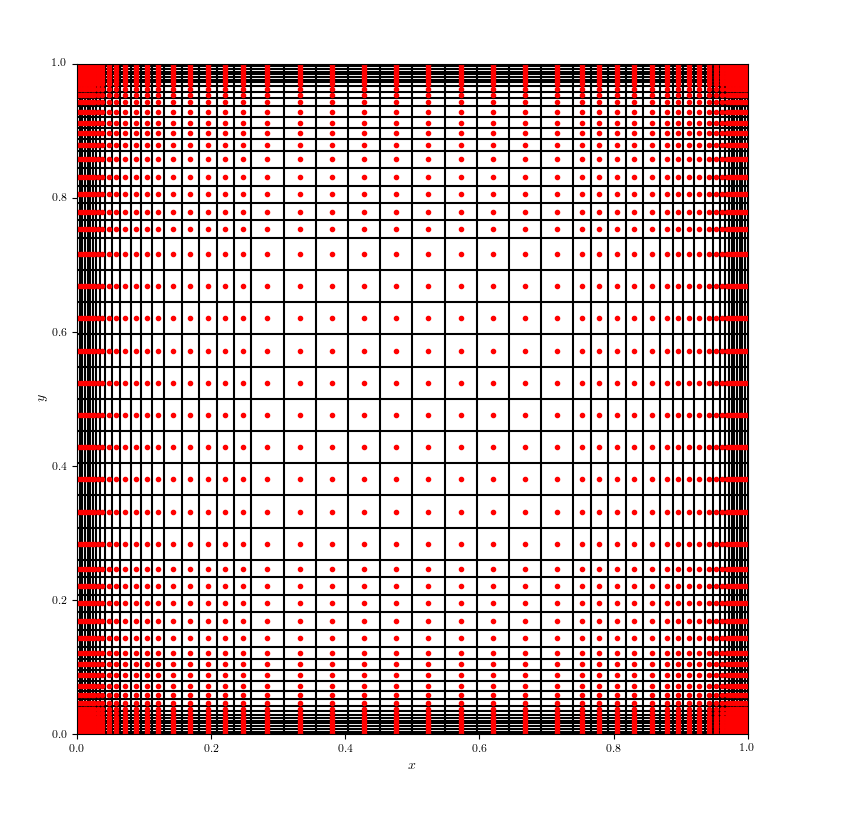}}\hskip20pt
\subfigure[Quadrature points in 3D]{\label{QPs_LDC_3D}
\includegraphics [width=0.4\textwidth]{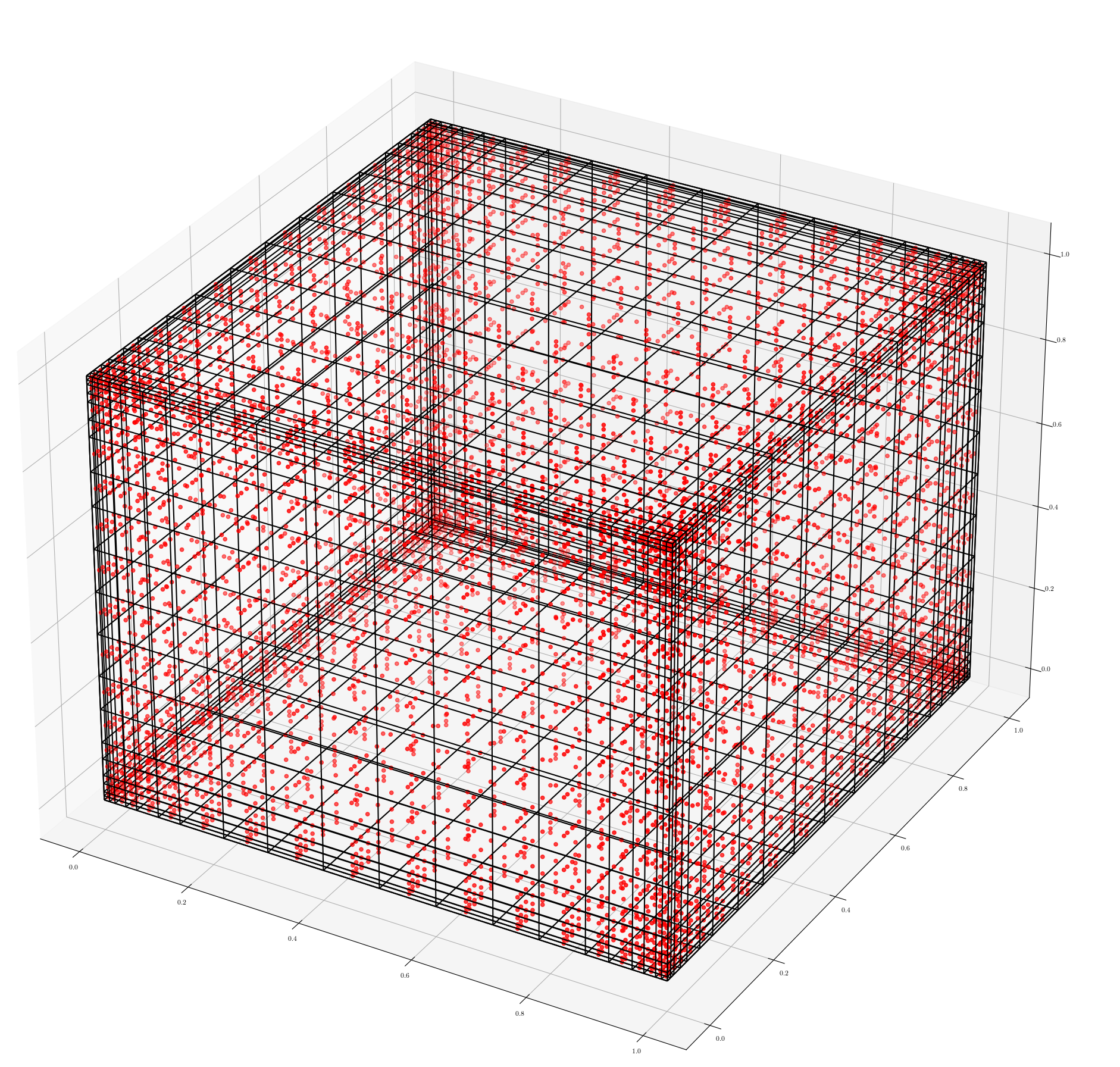}}
\caption{Local refined quadrature points of {lid-driven} cavity test case used for VPVnet method.}
\label{fig:testcase 5.1}
\end{figure}

The main difficulty of this problem comes from the presence of singularities at corners of the domain.
The pressure and the vorticity are not finite at the two top corners of the domain.
Singularities also present at the corners $(0,0)$ and $(1,0)$, where the second derivative of the velocity is unbounded, which is weaker than the previous ones.
By calculating the norm, one can see that $\bm{g} \in H^{1/2-\delta}(\partial\Omega, R^2)$ for all $0<\delta<1/2$
but $\bm{g}\notin H^{1/2}$, hence the velocity is not in $H^1$.
Thus the usual variational formulation for the Stokes' problem does not apply for the {lid-driven} cavity {problems}.
However, this test case is used in many papers to test finite elements methods and the discontinuous boundary condition is usually applied directly to the finite element discretization in practice.
In fact, on one hand, we can approximate the discontinuous boundary data by a smooth one and then apply a finite element method to the regularized problem,
on the other hand, it has been observed that for properly defined continuous boundary condition, the discrete system will not see the difference between the discontinuous and the continuous boundary condition \cite{Caizhiqiang2020,DuranGastaldiLombardi_2020}.
In this paper, we expect to explore the capacity of the VPVnet method to solve this problem.

\begin{figure}[!h]
\centering
\subfigure[Speed]{
\label{Fig5.sub.1}
\includegraphics [width=0.30\textwidth]{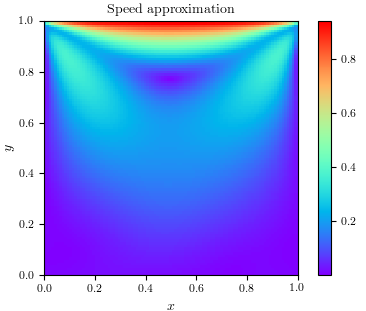}
}
\subfigure[Pressure]{
\label{Fig5.sub.2}
\includegraphics [width=0.30\textwidth]{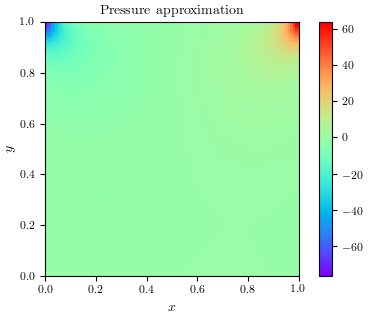}
}
\subfigure[Pressure plot in 3D ]{
\label{Fig5.sub.3}
\includegraphics [width=0.30\textwidth]{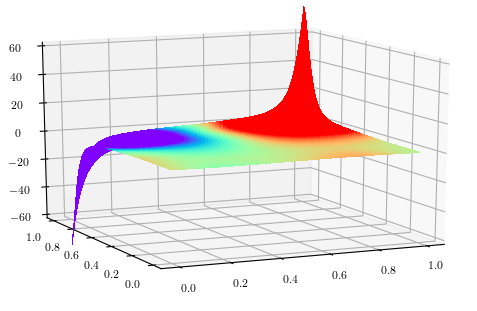}
}\\
\subfigure[Velocity ]{
\label{Fig5.sub.4}
\includegraphics [width=0.28\textwidth]{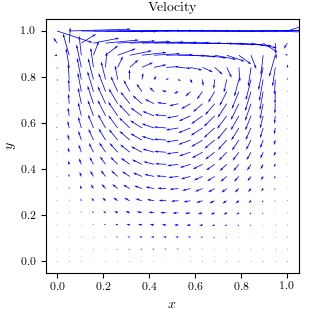}
}
\subfigure[Streamline ]{
\label{Fig5.sub.5}
\includegraphics [width=0.315\textwidth]{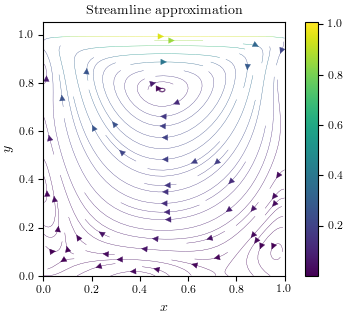}
}

\caption{Numerical results of {lid-driven} cavity test case obtained with VPVnet method using $50\times50$ locally refine quadrature points in $\Omega$.}
\label{fig:testcase 5.2}
\end{figure}
We use the following computation configuration for 2D case {computation}:
the parameters $\alpha=1$ and a neural network with $12 \times 16$ hidden layers and
neurons is employed.
A locally refined partition of size $50\times 50$ which is geometrically refined towards all the corners, i.e. $(0,0)$, $(1,0)$, $(0,1)$ and $(1,1)$, are considered.
The locally refined partition has elements of size varies from $4.0e-3$ to $4.8e-2$, an illustration is shown in Fig \ref{QPs_LDC_2D}.
One point quadrature rules is employed.
We pick the value of $h=1/50$.
As for the boundary condition of $u$, $v$, the partition is in accordance with the partitions interior domain,
i.e. locally refined points on the boundaries of $\Omega$ are used.
The discontinuous boundary condition is applied directly.

The numerical approximations obtained with the VPVnet method of the speed, velocity, pressure and the streamline for the {lid-driven} cavity test case are plotted in Figure \ref{fig:testcase 5.2}.
The exact solution of lid-driven problem is not known for the Stokes' problem.
The shape of streamlines is similar to the result given in \cite{LiRuo_2019}, strong pressure singularities are present at the top two corners, the results look quite reasonable.
Note that, there are some perturbations which can be observed near the domain boundaries, particularly in the figure for the streamline \ref{Fig5.sub.5}.
We refine the partition of the domain by $200\times200$ rectangular element with size from $1.0e-3$ to $8.0e-3$.
The perturbations turns to diminish (see Fig. \ref{fig:testcase 5.3}).

\begin{figure}[!h]
\centering
\subfigure[Speed]{
\label{Fig53.sub.1}
\includegraphics [width=0.30\textwidth]{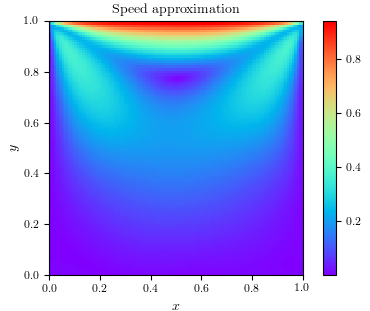}
}
\subfigure[Pressure]{
\label{Fig53.sub.2}
\includegraphics [width=0.30\textwidth]{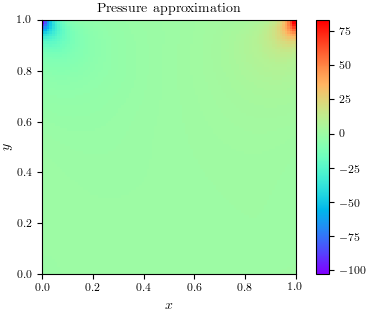}
}
\subfigure[Pressure plot in 3D ]{
\label{Fig53.sub.3}
\includegraphics [width=0.30\textwidth]{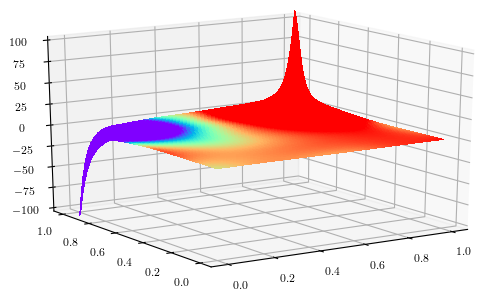}
}\\
\subfigure[Velocity ]{
\label{Fig53.sub.4}
\includegraphics [width=0.265\textwidth]{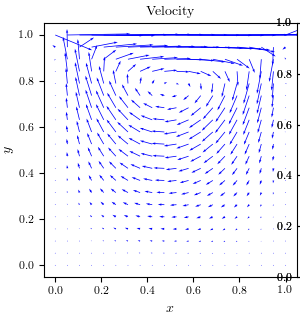}
}
\subfigure[Streamline ]{
\label{Fig53.sub.5}
\includegraphics [width=0.310\textwidth]{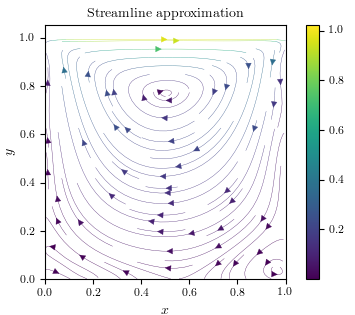}
}

\caption{Numerical results of {lid-driven} cavity test case obtained with VPVnet method using $200\times200$ locally refine quadrature points in $\Omega$.}
\label{fig:testcase 5.3}
\end{figure}
We also consider the 3D {lid-driven} cavity test case in a unit cubic domain.
A tangential velocity $\bm{u}_{\Gamma} =(1, 0, 0)^T $ is imposed at the top surface of the domain, while
$\bm{u} = (0, 0,0)^T$ is imposed on the rest of the boundary.
The parameters $\alpha=1$ and a neural network with $16 \times 16$ hidden layers and
neurons is employed for the computation.
Locally refined partitions of size $10\times 10\times10$ and $20\times 20\times20$ which are geometrically refined towards the domain corners are considered.
The locally refined partition has elements of size varies from $2.0e-2$ to $2.0e-1$ for the $10\times 10\times10$ partition, and $1.0e-2$ to $1.0e-1$ for $20\times 20\times20$ partition, an illustration is shown in Fig. \ref{QPs_LDC_3D}.
Note that for 3D computation, the BFGS algorithm is used to finetune the results obtained with 2000 Adam optimization iterations.
The numerical results are shown in Figures \ref{fig:testcase 5.4} and \ref{fig:testcase 5.5}, where Figure \ref{fig:testcase 5.5} show the top view along the $x,y,z$ axis respectively.
The numerical results indicate that the VPVnet method is efficient and accurate.

\begin{figure}
\centering
\subfigure[velocity with 1000 QPs]{
\label{Fig54.sub.1}
\includegraphics [width=0.45\textwidth]{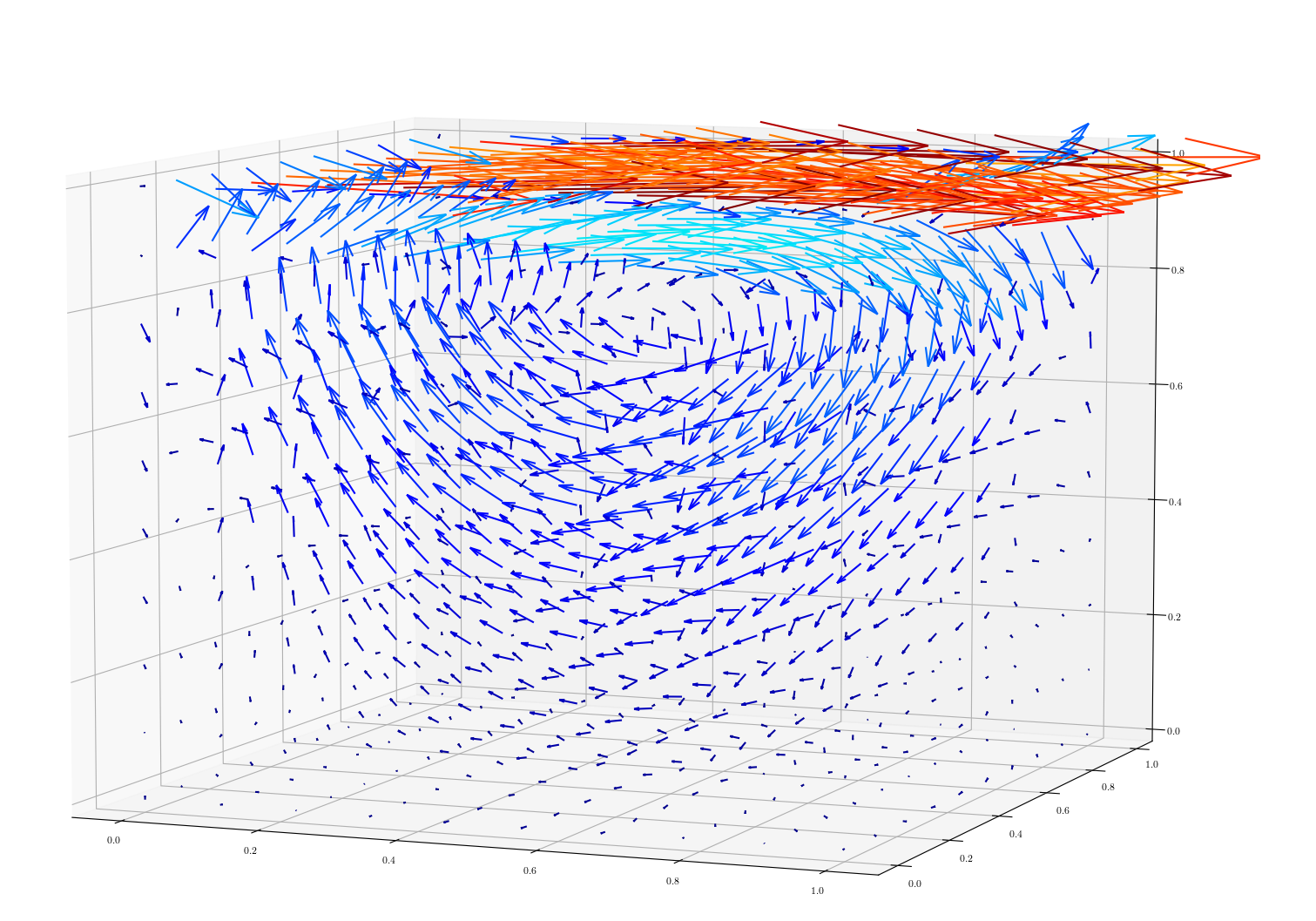}
}
\subfigure[velocity with 8000 QPs]{
\label{Fig54.sub.2}
\includegraphics [width=0.45\textwidth]{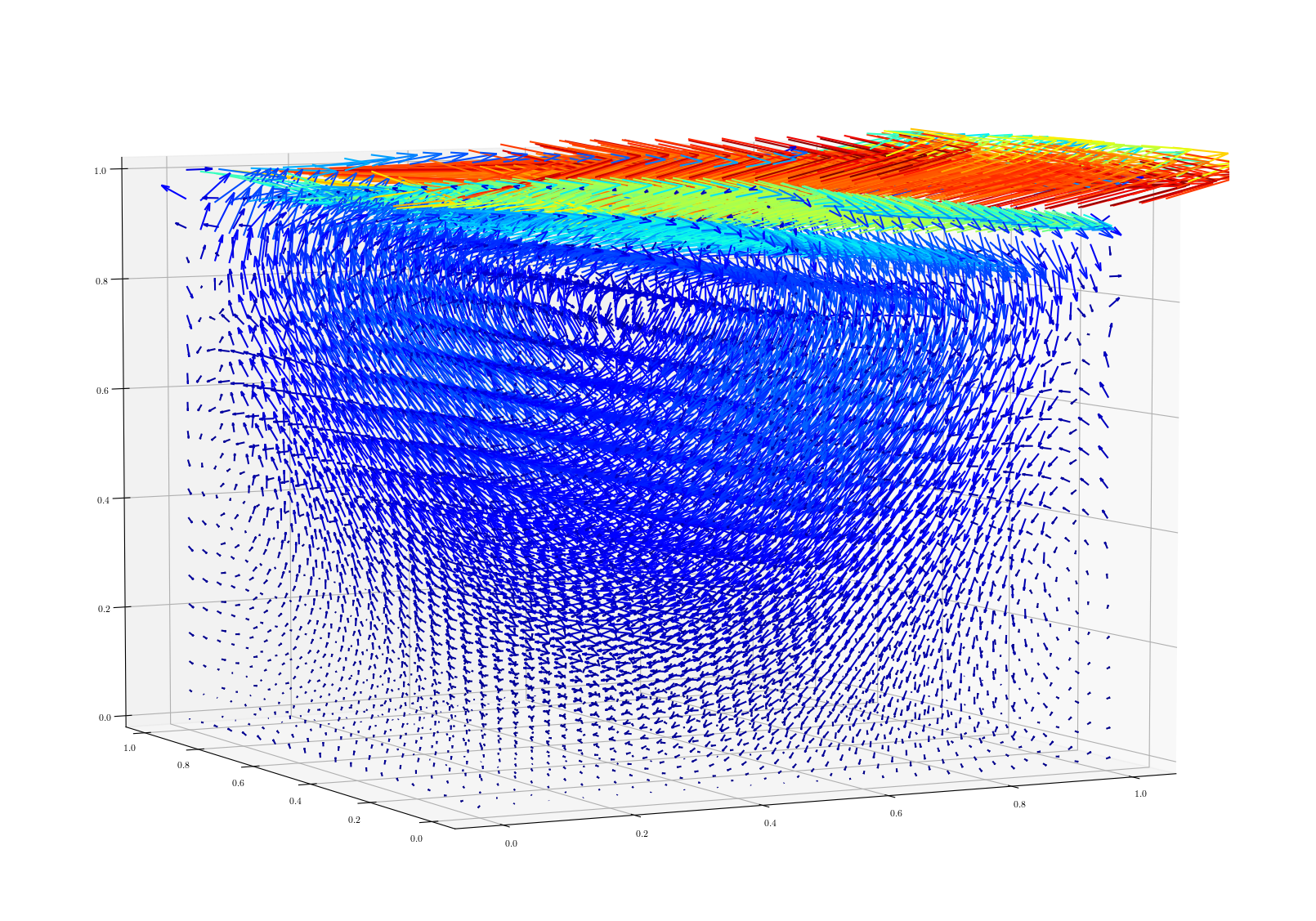}
}
\caption{Numerical results of 3D {lid-driven} cavity test case obtained with VPVnet method using $1000$ and $8000$ locally refine quadrature points in $\Omega$.}
\label{fig:testcase 5.4}
\end{figure}

\begin{figure}
\centering
\subfigure[view from x-axis]{
\label{Fig55.sub.3}
\includegraphics [width=0.32\textwidth]{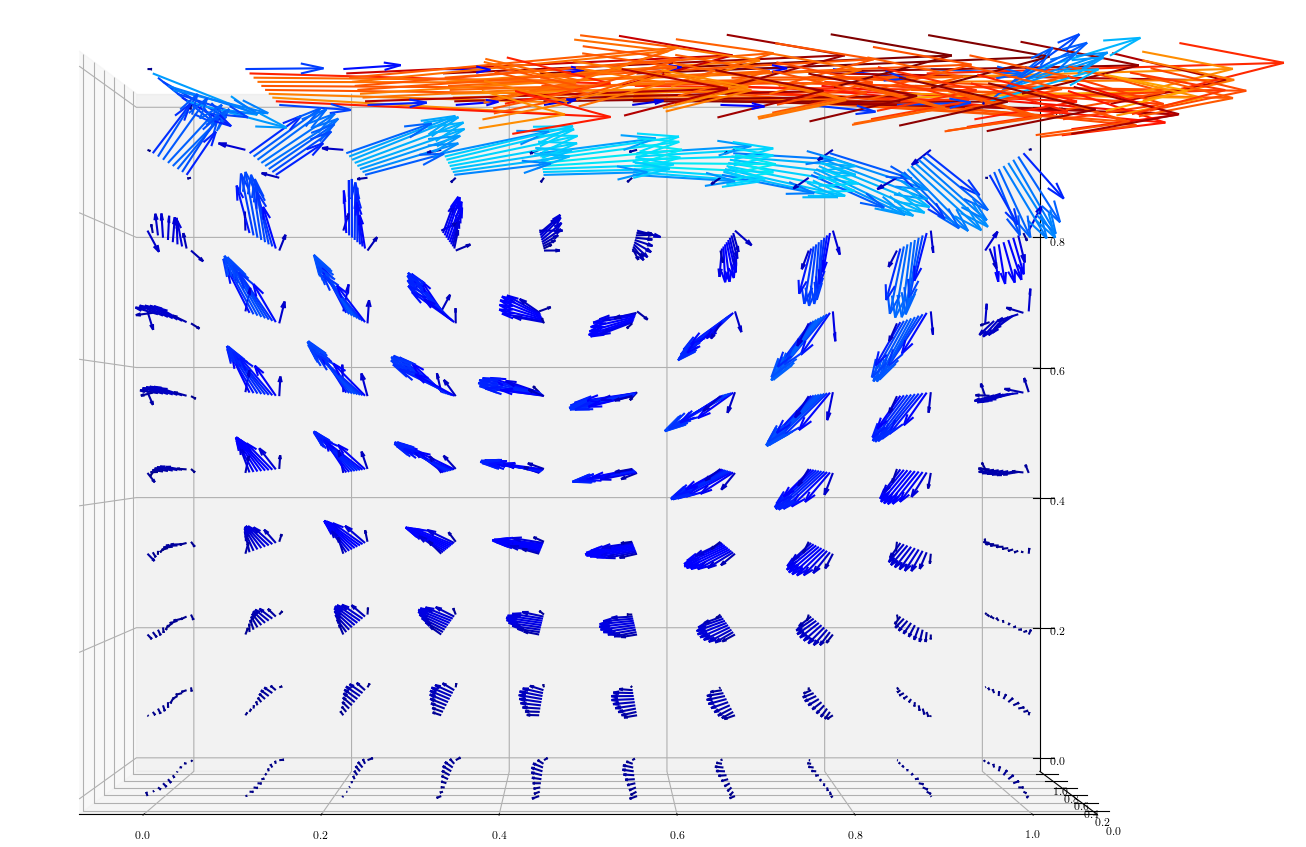}
}
\subfigure[view from y-axis]{
\label{Fig55.sub.1}
\includegraphics [width=0.28\textwidth]{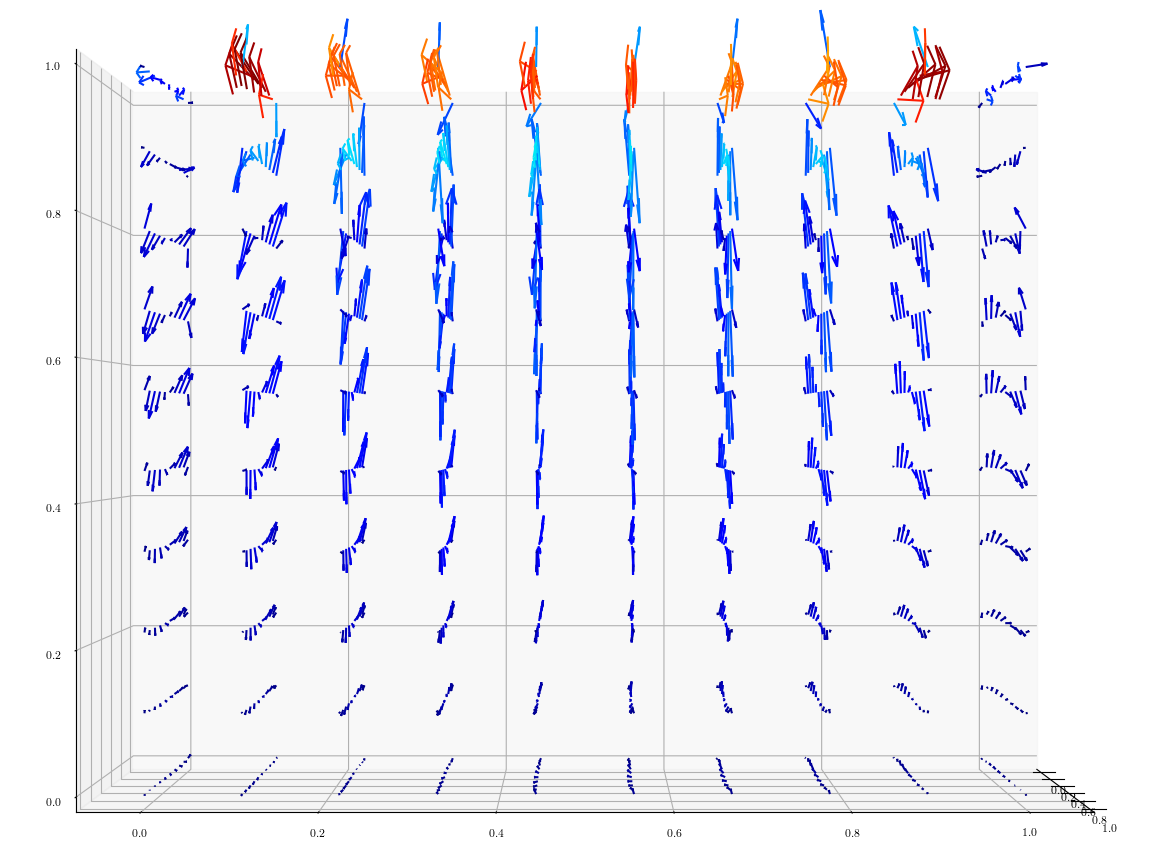}
}
\subfigure[view from z-axis]{
\label{Fig55.sub.2}
\includegraphics [width=0.31\textwidth]{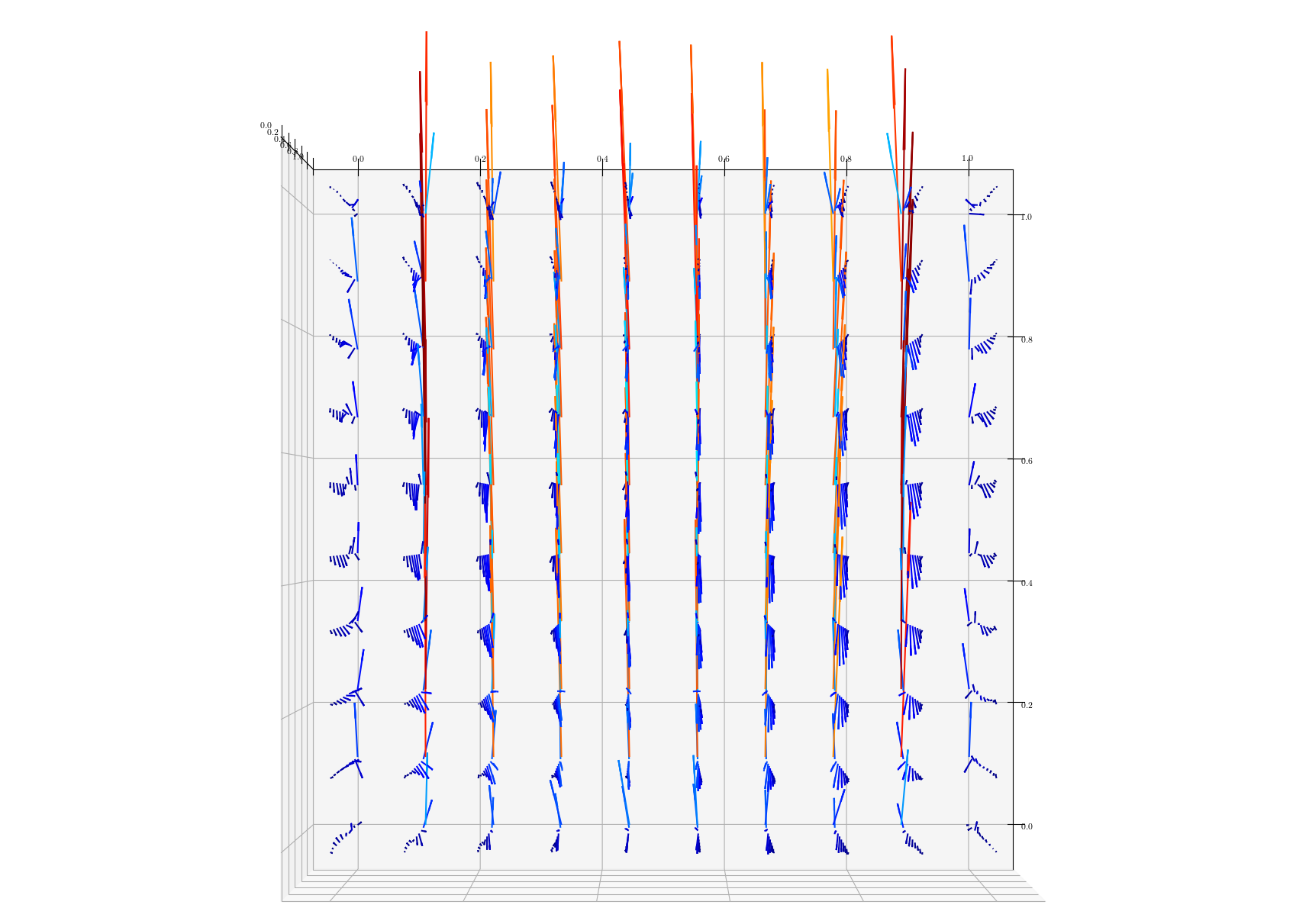}
}
\\
\subfigure[view from x-axis]{
\label{Fig55.sub.6}
\includegraphics [width=0.320\textwidth]{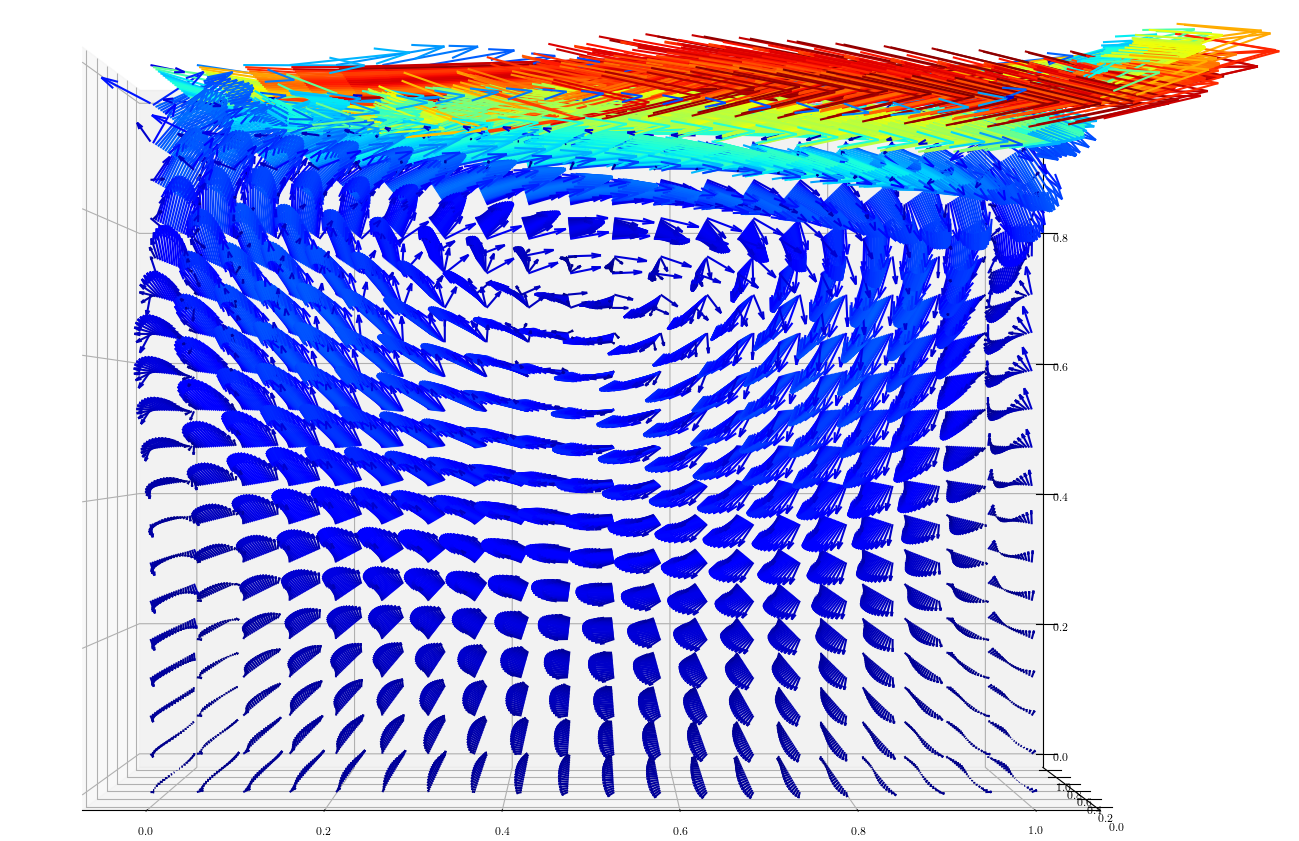}
}
\subfigure[view from y-axis]{
\label{Fig55.sub.4}
\includegraphics [width=0.28\textwidth]{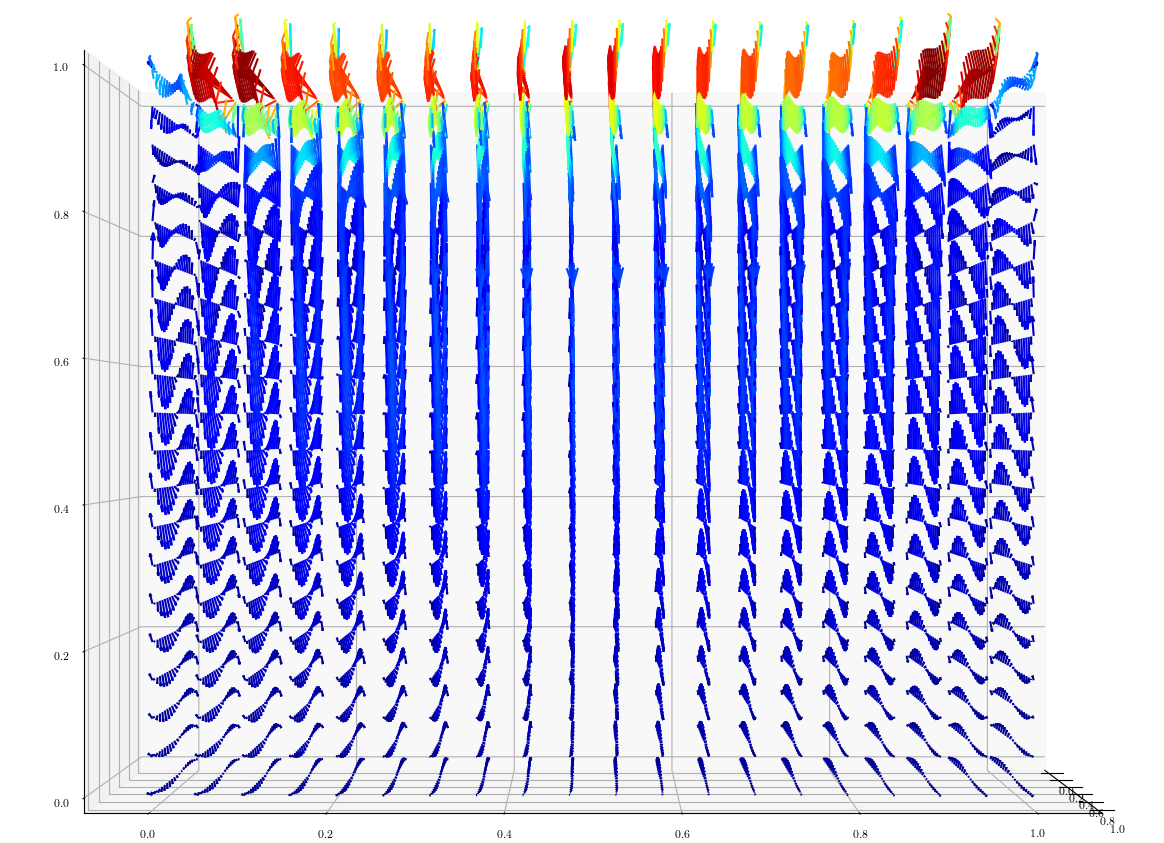}
}
\subfigure[view from z-axis]{
\label{Fig55.sub.5}
\includegraphics [width=0.31\textwidth]{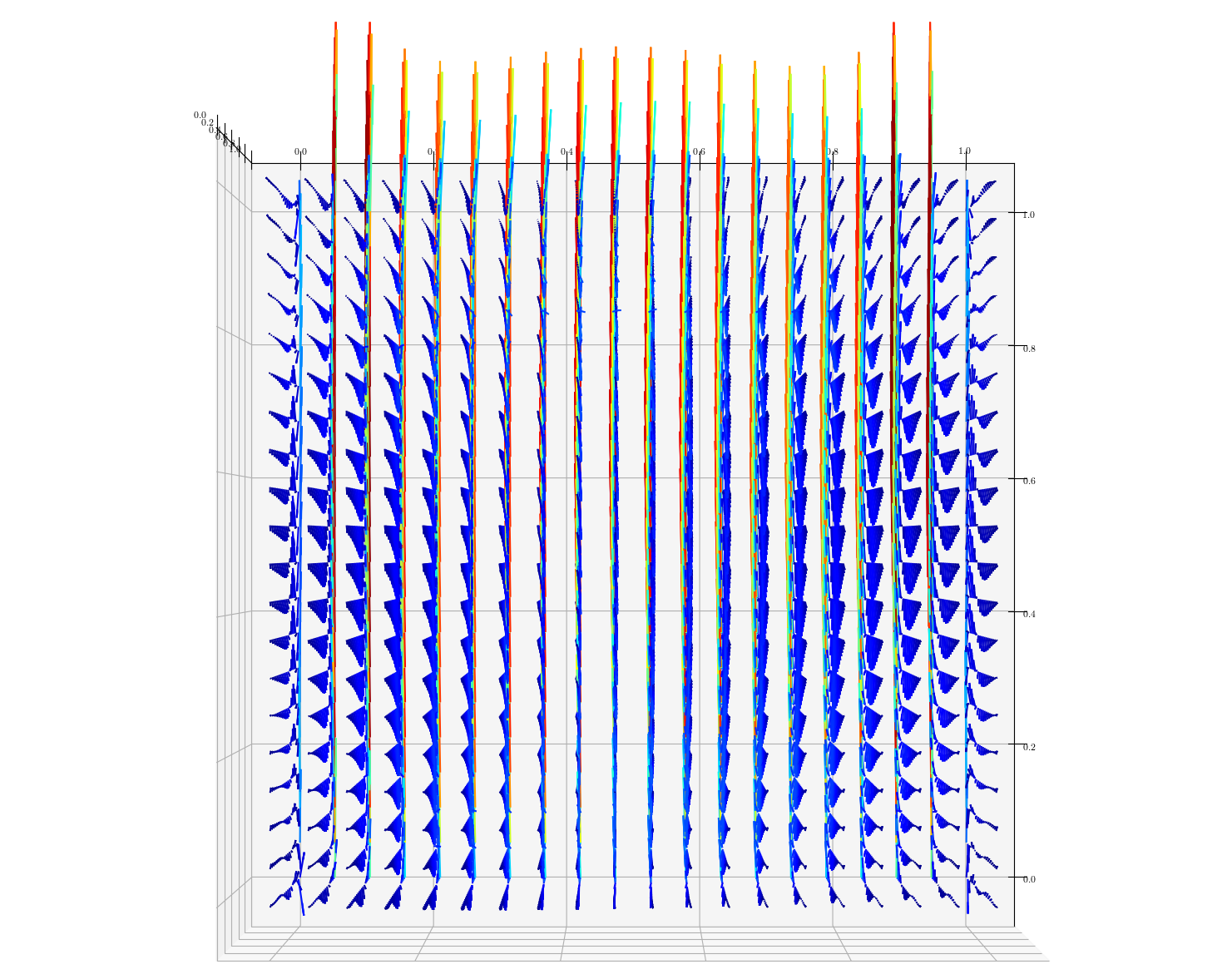}
}
\caption{Numerical results of {lid-driven} cavity test case obtained with VPVnet method using $1000$ and $8000$ locally refine quadrature points in $\Omega$.}
\label{fig:testcase 5.5}
\end{figure}
\section{Discussion and conclusion}\label{SectionDisCon}
A deep neural network method is proposed and verified to approximate the solution of the Stokes' equations.
The method make{s} use of least square functionals based on the {first-order} VPV formulation \eqref{EQ:Velocity-Pressure-Vorticity_formulation} of the Stokes' equations as loss functions.
It has less regularity requirements on the solutions compared with other methods such as the methods based on the original form of the PDEs \cite{DGM_2020,raissi2019physics}.
{Convergence and error estimates have been established for this method.}
We observed that smooth solutions of the Stokes' equations can be well approximated by the VPVnet method using Resnet neural networks
with a rather small number of neurons and hidden layers.
Within a certain limit, a better precision can be obtained by using more hidden layers, more neurons and more quadrature points.
Either smooth activation function, such as $Sin(x)$, $Tanh$, $Sigmoid$, etc can be employed and all of them can provide satisfying approximations.
The method is divergence-free and pressure-robust.

We focus on the ability of the method for the approximation of Stokes' equations with non-smooth/singular solutions, i.e. $u\notin H^2$ and $p\notin H^1$.
For the velocity variable $\bm{u}\in H^s$, $2>s>1$ ($\in C^0$), the VPVnet methods can approximate it very well,
which is in consistent with our theoretical analysis.
As for the pressure variable $p\in H^{s-1}$, $2>s>1$ ($\notin C^0$),
the approximation turn to be more difficult.
Refinement can help to improve the precision, while more efforts is needed to reach a precision satisfying.
The future work will be aimed at further investigation on this aspect.
{In the meantime, a variant of the proposed VPVnet method is under investigation}, namely the {VVGPnet} method
which is based on
the first-order system of velocity/velocity gradient/pressure (VVGP) formulation \cite{Liukuo_2012} of the Stokes' equations.
The VVGP system also provides a $H^1$ least square functional and can be approximated similarly as the VPVnet method.
The VVGPnet method is quite suitable for the approximation of the Stokes interface problems \cite{WangKhoo_2013}, which usually has jumps of the gradient of the velocity.
The method will be detailed in a forthcoming paper dedicated for the interface problems.

The method presented here has already been extended to the incompressible Navier-Stokes Equations and will be discussed in {a forthcoming} paper.

\section*{Acknowledgments}
The authors would like to thank Hailong Sheng for helpful discussions. 
The research of Liu was partially supported by China National Natural Science Foundation (No. 12001306), Guangdong Provincial Natural Science Foundation (No. 2017A030310285).
The research of Yang (corresponding author) was funded in part by Beijing Academy of Artificial Intelligence.


\end{document}